\documentclass[11pt]{article}
\usepackage{fullpage}

\usepackage[utf8]{inputenc}
\usepackage{amsmath}
\usepackage{amsfonts}
\usepackage{dsfont}
\usepackage{natbib}
\usepackage{xcolor}
\usepackage{amssymb}
\usepackage{algorithm}
\usepackage{algorithmic}
\usepackage{amsthm}
\usepackage{graphicx}
\usepackage{url}
\usepackage{enumitem}
\usepackage{microtype}
\usepackage{subcaption}
\usepackage{booktabs}

\usepackage{hyperref}

\newcommand{\rel}{{F / \phi}}
\newtheorem{theorem}{Theorem}
\newtheorem{assumption}{Assumption}
\newtheorem{definition}{Definition}
\newtheorem{remark}{Remark}
\newcommand{\R}{\mathbb{R}}
\newcommand{\esp}[1]{\mathbb{E}\left[ #1 \right]}

\newcommand{\minimize}{\mathop{\rm minimize}}

\newcommand{\dom}{{\mathrm{dom\,}}}

\newcommand{\BEAS}{\begin{eqnarray*}}
\newcommand{\EEAS}{\end{eqnarray*}}
\newcommand{\BEA}{\begin{eqnarray}}
\newcommand{\EEA}{\end{eqnarray}}
\newcommand{\BEQ}{\begin{equation}}
\newcommand{\EEQ}{\end{equation}}
\newcommand{\BIT}{\begin{itemize}}
\newcommand{\EIT}{\end{itemize}}
\newcommand{\BNUM}{\begin{enumerate}}
\newcommand{\ENUM}{\end{enumerate}}
\newcommand{\BA}{\begin{array}}
\newcommand{\EA}{\end{array}}

\newcommand{\argmin}{\mathop{\rm argmin}}

\newcommand{\dE}{\mathbb {E}}

\newcommand{\dP}{\mathbb{P}}

\newcommand{\dR}{\mathbb {R}}

\newcommand{\cB}{\mathcal {B}}

\newcommand{\cN}{\mathcal {N}}

\newcommand{\cS}{\mathcal {S}}

\newcommand{\opn}[1]{{{\left\lVert#1\right\rVert_{op}}}}
\newcommand{\second}{{\prime \prime}}

\newtheorem{lemma}{Lemma}

\newtheorem{corollary}{Corollary}

\def \E{{\mathbb E}}

\def \X{{\mathcal X}}
\def \Y{{\mathcal Y}}

\title{Statistically Preconditioned Accelerated Gradient Method\\ for  Distributed Optimization}

\author{
Hadrien Hendrikx\thanks{D\'epartement d’informatique de l’ENS, ENS, CNRS, PSL University, Paris, France} \footnotemark[3]\\
(\texttt{hadrien.hendrikx@inria.fr})
\and
Lin Xiao\thanks{Machine Learning and Optimization Group, 
Microsoft Research, Redmond, WA.}\\
(\texttt{lin.xiao@microsoft.com})
\and 
S\'ebastien Bubeck\footnotemark[2]\\
(\texttt{sebubeck@microsoft.com})
\and
Francis Bach\footnotemark[1] \thanks{INRIA, Paris, France}\\
(\texttt{francis.bach@inria.fr}) 
\and
Laurent Massouli\'e\footnotemark[1] \footnotemark[3]\\
(\texttt{laurent.massoulie@inria.fr})
}
\date{}

\begin{document}

\maketitle

\begin{abstract}
We consider the setting of distributed empirical risk minimization where multiple machines compute the gradients in parallel and a centralized server updates the model parameters.
In order to reduce the number of communications required to reach a given accuracy, we propose a \emph{preconditioned} accelerated gradient method where the preconditioning is done by solving a local optimization problem over a subsampled dataset at the server.
The convergence rate of the method depends on the square root of the relative condition number between the global and local loss functions.
We estimate the relative condition number for linear prediction models by studying \emph{uniform} concentration of the Hessians over a bounded domain, which allows us to derive improved convergence rates for existing preconditioned gradient methods and our accelerated method.
Experiments on real-world datasets illustrate the benefits of acceleration in the ill-conditioned regime.
\end{abstract}

\section{Introduction}

We consider empirical risk minimization problems of the form
\begin{equation}\label{eqn:reg-ERM}
    \minimize_{x\in\R^d}~~ \Phi(x)\triangleq F(x)+\psi(x),
\end{equation} 
where $F$ is the empirical risk over a dataset $\{z_1,\ldots,z_N\}$:
\begin{equation}\label{eqn:average-loss}
    F(x) = \frac{1}{N}\sum_{i=1}^N \ell(x,z_i),
\end{equation}
and $\psi$ is a convex regularization function. 
We incorporate smooth regularizations such as squared Euclidean norms $(\lambda/2)\|x\|^2$
into the individual loss functions $\ell(x,z_i)$, 
and leave $\psi$ mainly for non-smooth regularizations
such as the $\ell_1$-norm or the indicator function of a constraint set.

In modern machine learning applications, the dataset is often very large and 
has to be stored at multiple machines.
For simplicity of presentation, we assume $N=mn$, where $m$ is the number
of machines and $n$ is the number of samples stored at each machine.
Let $\mathcal{D}_j=\{z^{(j)}_1,\ldots,z^{(j)}_n\}$ denote the dataset
at machine~$j$ and define the local empirical risk 
\begin{equation}\label{eqn:fj-def}
    f_j(x) = \frac{1}{n}\sum_{i=1}^n \ell\bigl(x, z^{(j)}_i\bigr),
    \quad j=1,\ldots,m.
\end{equation}
The overall empirical risk of Equation~\eqref{eqn:average-loss} can then be written as
\[
    F(x) = \frac{1}{m}\sum_{j=1}^m f_j(x) = 
    \frac{1}{nm}\sum_{j=1}^m\sum_{i=1}^n \ell\bigl(x, z^{(j)}_i\bigr).
\]
We assume that $F$ is $L_F$-smooth and $\sigma_F$-strongly convex 
over $\dom\psi$, in other words,
\begin{equation}\label{eqn:classic-ssc}
  \sigma_F I_d \preceq \nabla^2 F(x) \preceq L_F I_d, \quad \forall x\,\in\dom\psi,
\end{equation}
where $I_d$ is the $d\times d$ identity matrix.
The condition number of $F$ is defined as $\kappa_F=L_F/\sigma_F$.

We focus on a basic setting of distributed optimization where the~$m$ machines
(workers) compute the gradients in parallel and a centralized server updates the
variable~$x$.
Specifically, during each iteration $t=0,1,2,\ldots$,
\begin{enumerate}[label=(\roman*)]
    \item the server broadcasts $x_t$ to all $m$ machines;
    \item each machine~$j$ computes the gradient $\nabla f_j(x_t)$ and sends 
it back to the server;
    \item the server forms $\nabla F(x_t) = \frac{1}{m}\sum_{j=1}^m \nabla f_j(x_t)$ and uses it to compute the next iterate $x_{t+1}$.
\end{enumerate}
A standard way for solving problem~\eqref{eqn:reg-ERM} in this setting
is to implement the proximal gradient method at the server:
\begin{equation}\label{eqn:prox-grad}
x_{t+1} \!=\! \argmin_{x \in \R^d}\Bigl\{\nabla F(x_t)^\top\! x \!+\! \psi(x)\!+\!\frac{1}{2\eta_t}\|x\!-\!x_t\|^2\Bigr\},
\end{equation}
where $\|\cdot\|$ denotes the Euclidean norm and $\eta_t>0$ is the step size.
Setting $\eta_t=1/L_F$ leads to linear convergence:
\begin{equation}\label{eqn:prox-grad-rate}
    \Phi(x_t) - \Phi(x_*) \leq \left(1-\kappa_F^{-1}\right)^t \textstyle\frac{L_F}{2} \|x_*-x_0\|^2,
\end{equation}
where $x_*=\arg\min \Phi(x)$ \citep[e.g.,][Section~10.6]{Beck17book}.
In other words, in order to reach $\Phi(x_t)-\Phi(x_*)\leq\epsilon$, we need
$O(\kappa_F\log(1/\epsilon))$ iterations, which is also the number of
communication rounds between the workers and the server.
If we use accelerated proximal gradient methods 
\citep[e.g.,][]{Nesterov04book, BeckTeboulle09fista, Nesterov13composite}
at the server, then the iteration/communication complexity can be improved to
$O(\sqrt{\kappa_F}\log(1/\epsilon))$.

\subsection{Statistical Preconditioning}
\label{sec:stat-precond}

In general, for minimizing $F(x)=(1/m)\sum_{j=1}^m f_j(x)$ with first-order 
methods, the communication complexity of $O(\sqrt{\kappa_F}\log(1/\epsilon))$ 
cannot be improved \citep{ArjevaniShamir2015NIPS,ScamanBach2017}.
However, for distributed empirical risk minimization (ERM), the additional 
finite-sum structure of each $f_j$ in~\eqref{eqn:fj-def} 
allows further improvement.
A key insight here is that if the datasets $\mathcal{D}_j$ at different workers
are i.i.d.~samples from the same source distribution, then the local empirical losses $f_j$ are statistically very similar to each other and to their average~$F$, especially when~$n$ is large.
\emph{Statistical preconditioning} is a technique to further reduce 
communication complexity based on this insight.

An essential tool for preconditioning in first-order methods
is the Bregman divergence. 
The Bregman divergence of a strictly convex and differentiable function $\phi$ is defined as
\begin{equation}\label{eqn:Bregman-def}
    D_\phi(x, y) \triangleq \phi(x) - \phi(y) - \nabla \phi(y)^\top(x - y).
\end{equation}
We also need the following concepts of relative smoothness and strong convexity
\cite{bauschke2017descent,lu2018relatively}.
\begin{definition}
\label{def:rel-ssc}
Suppose $\phi:\R^d\to\R$ is convex and twice differentiable.
The function~$F$ is said to be $L_{F/\phi}$-smooth and 
$\sigma_{F/\phi}$-strongly convex with respect to~$\phi$ if for all $x\in\R^d$,
\begin{equation}\label{eqn:rel-ssc}
    \sigma_{F/\phi} \nabla^2\phi(x) \preceq \nabla^2 F(x) \preceq L_{F/\phi} \nabla^2\phi(x). 
\end{equation}
\end{definition}
The classical definition in~\eqref{eqn:classic-ssc} can be viewed
as relative smoothness and strong convexity where $\phi(x)=(1/2)\|x\|^2$.
Moreover, it can be shown that~\eqref{eqn:rel-ssc} holds if and only if for all $x,y \in \R^d$
\begin{equation}\label{eqn:rel-ssc-div}
    \sigma_{F/\phi} D_\phi(x,y) \leq D_F(x,y) \leq L_{F/\phi} D_\phi(x,y).
\end{equation}
Consequently, we define the relative condition number of $F$ with respect to $\phi$ as $\kappa_{F/\phi}=L_{F/\phi}/\sigma_{F/\phi}$.

Following the Distributed Approximate Newton (DANE) method 
by \citet{shamir2014communication}, 
we construct the reference function $\phi$ by
adding some extra regularization to one of the local loss functions
(say $f_1$, without loss of generality):
\begin{equation}\label{eqn:reference-def}
    \phi(x) = f_1(x) + \frac{\mu}{2}\|x\|^2.
\end{equation}
Then we replace $(1/2)\|x-x_t\|^2$ in the proximal gradient method~\eqref{eqn:prox-grad} with the Bregman divergence of~$\phi$, i.e., 
\begin{equation}\label{eqn:rel-prox-grad}
    x_{t+1}\! =\! \argmin_{x\in \R^d} \Bigl\{\nabla F(x_t)^\top\! x + \psi(x) + \frac{1}{\eta_t} D_\phi(x,x_t) \Bigr\}.
\end{equation}
In this case, worker~1 acts as the server to compute $x_{t+1}$, which requires solving a nontrivial optimization problem involving the local loss function~$f_1$.

According to \citet{shamir2014communication} and \citet{lu2018relatively}, 
with $\eta_t=1/L_{F/\phi}$, the sequence
$\{x_t\}$ generated by~\eqref{eqn:rel-prox-grad} satisfies
\begin{equation}\label{eqn:rel-grad-rate}
    \Phi(x_t) - \Phi(x_*) \leq \bigl(1-\kappa_{F/\phi}^{-1}\bigr)^t 
    L_{F/\phi} D_\phi(x_*, x_0),
\end{equation}
which is a direct extension of~\eqref{eqn:prox-grad-rate}. 
Therefore, the effectiveness of preconditioning hinges on how much smaller
$\kappa_{F/\phi}$ is compared to~$\kappa_F$.
Roughly speaking, the better $f_1$ or $\phi$ approximates~$F$, 
the smaller $\kappa_{F/\phi}$ ($\geq 1$) is. 
In the extreme case of $f_1\equiv F$ (with only one machine $m=1$), we can choose $\mu=0$ and thus
$\phi\equiv F$, which leads to $\kappa_{F/\phi}=1$, and we obtain
the solution within one step.

In general, we choose~$\mu$ to be an upper bound on the spectral norm of
the matrix difference $\nabla^2 f_1-\nabla^2 F$. 
Specifically, we assume that \emph{with high probability}, for the operator norm between matrices (i.e., the largest singular value),
\begin{equation}\label{eqn:Hessian-approx-mu}
    \left\|\nabla^2 f_1(x) - \nabla^2 F(x)\right\| \leq \mu, 
    \quad \forall\, x\in\dom\psi,
\end{equation}
which implies \citep[][Lemma~3]{ZhangXiao2018DiSCO},
\begin{equation}\label{eqn:rel-sigma-mu}
    \frac{\sigma_F}{\sigma_F+2\mu} \nabla^2\phi(x) 
    \preceq \nabla^2 F(x) \preceq \nabla^2\phi(x).
\end{equation}
Now we invoke a statistical argument based on the empirical average structure
in~\eqref{eqn:fj-def}.
Without loss of generality, we assume that $\mathcal{D}_1$ contains the first~$n$ samples of $\{z_1,\ldots,z_N\}$ and thus
$\nabla^2 f_1(x) = \frac{1}{n}\sum_{i=1}^n \nabla^2\ell(x,z_i)$.
For any fixed~$x$, we can use Hoeffding's inequality for 
matrices \citep{tropp2015introduction} to obtain,
with probability $1-\delta$,
\begin{equation}\label{eqn:matrix-hoeffding}
    \biggl\|\frac{1}{n}\!\sum_{i=1}^n \!\nabla^2\ell(x,z_i)\!-\!\nabla^2 F(x)\biggr\|
    \!\leq\! \sqrt{\frac{32 L_\ell^2 \log(d/\delta)}{n}},
\end{equation}
where $L_\ell$ is the uniform upper bound on $\|\nabla^2\ell(x,z_i)\|$. 

If the losses $\ell(x,z_i)$ are \emph{quadratic} in~$x$, then the Hessians are 
constant and~\eqref{eqn:Hessian-approx-mu} holds with 
$\mu=\widetilde{O}(L_\ell/\sqrt{n})$, hiding the factor $\log(d/\delta)$.
In this case, we derive from~\eqref{eqn:rel-sigma-mu} that
\begin{equation}\label{eqn:rel-cond-quadratic}
\kappa_{F/\phi}=1 + \frac{2\mu}{\sigma_F}
=1+\widetilde{O}\left(\frac{\kappa_\ell}{\sqrt{n}}\right),
\end{equation}
where we assume $\sigma_F\approx \sigma_\ell$, where $\nabla^2 \ell(x, z_i) \succeq \sigma_\ell I_d$ for all $x$.
Therefore,  for large~$n$, whenever we have $\kappa_{F/\phi}<\kappa_F$,
the communication complexity $O(\kappa_{F/\phi}\log(1/\epsilon))$
is better than without preconditioning.

For non-quadratic loss functions, we need to ensure that \eqref{eqn:Hessian-approx-mu} holds \emph{uniformly} over a compact domain with high probability.
Standard ball-packing arguments encounter an additional factor of $\sqrt{d}$
\citep[e.g.,][Lemma~6]{ZhangXiao2018DiSCO}.
In this case, we have $\mu=\widetilde{O}(L_\ell\sqrt{d/n})$ and 
\begin{equation}\label{eqn:rel-cond-non-quadratic}
    \kappa_{F/\phi}=1 + \frac{2\mu}{\sigma_F}
     =1+\widetilde{O}\biggl(\frac{\kappa_\ell\sqrt{d}}{\sqrt{n}}\biggr),
 \end{equation}
which suggests that the benefit of preconditioning may degrade or disappear in high dimension.

\subsection{Contributions and Outline}
In this paper, we make the following two contributions.
    
First, we propose a Statistically Preconditioned Accelerated Gradient (SPAG) method that can further reduce the communication complexity. 
Accelerated methods with $O(\sqrt{\kappa_{F/\phi}}\log(1/\epsilon))$ complexity
have been developed for quadratic loss functions
(see related works in Section~\ref{sec:related-works}).
However,
\citet{dragomir2019optimal} have shown that acceleration is not possible in general in the relatively smooth and strongly convex setting, and that more assumptions are needed.
Here, by leveraging the fact the reference function~$\phi$ itself is smooth and strongly convex, we obtain
\[
    \Phi(x_t) - \Phi(x_*) \leq \prod_{\tau=1}^{t}\biggl(1-\frac{1}{\sqrt{\kappa_{F/\phi} G_\tau}} \biggr) L_{F/\phi} D_\phi(x_*, x_0),
\]
where $1 \leq G_t\leq \kappa_\phi$ and $G_t\to 1$ geometrically.
Moreover, $G_t$ can be calculated at each iteration and serve as numerical certificate of the actual convergence rate. 
In all of our experiments, we observe $G_t\approx 1$ even in early iterations,
which results in
$O(\sqrt{\kappa_{F/\phi}}\log(1/\epsilon))$ iterations empirically.

Second, we derive refined bounds on the relative condition number for linear prediction models. 
Linear models such as logistic regression have the form
$\ell(x, z_i) = \ell_i(a_i^\top x) + (\lambda/2)\|x\|^2$. 
Assume that $\ell''_i(a_i^\top x)\leq 1$ and $\|a_i\|\leq R$ for all~$i$,
which implies $L_\ell=R^2$ and $\kappa_\ell=R^2/\lambda$.
Then the Hoeffding bounds in~\eqref{eqn:rel-cond-quadratic} for quadratics becomes
$\kappa_{F/\phi}=1+\widetilde{O}\Bigl(\frac{R^2}{\sqrt{n}\lambda}\Bigr)$, 
and for nonquadratics, 
the bound in~\eqref{eqn:rel-cond-non-quadratic} (from previous work)
becomes
$\kappa_{F/\phi}=1+\widetilde{O}\Bigl(\frac{R^2\sqrt{d}}{\sqrt{n}\lambda}\Bigr)$.
We show that:
\begin{itemize}[itemsep=0pt,topsep=0pt,leftmargin=15pt]
    \item For quadratic losses, the bound on relative condition number can be improved by a factor of $\sqrt{n}$, i.e., 
        $$\kappa_{F/\phi}=\frac{3}{2} + O\left(\frac{R^2}{n\lambda}\log\left(\frac{d}{\delta}\right)\right).$$
    \item For non-quadratic losses, we derive a uniform concentration bound to remove the dependence of $\kappa_{F/\phi}$ on~$d$,
        $$\kappa_{F/\phi}=1 + O\left(\frac{R^2}{\sqrt{n}\lambda}\left(RD + \sqrt{\log(1/\delta)}\right)\right),$$
        where $D$ is the diameter of $\dom\phi$ (bounded domain).
        We also give a refined bound when the inputs $a_i$ are sub-Gaussian.
\end{itemize}
These new bounds on $\kappa_{F/\phi}$ improve the convergence rates for all existing accelerated and non-accelerated preconditioned gradient methods
(see related work in Section~\ref{sec:related-works}).

We start by discussing related work in Section~\ref{sec:related-works}. 
In Section~\ref{sec:spag}, we introduce SPAG and give its convergence analysis.
In Section~\ref{sec:concentration}, we derive sharp bounds on the relative condition number, and discuss their implications on the convergence rates of SPAG and other preconditioned gradient methods. We present experimental results in Section~\ref{sec:experiments}.

\section{Related Work}
\label{sec:related-works}

\citet{shamir2014communication} considered the case  $\psi\equiv 0$ and
introduced the statistical preconditioner~\eqref{eqn:reference-def} in DANE.
Yet, they define a separate $\phi_j(x)=f_j(x)+(\mu/2)\|x\|^2$ for each worker~$j$, compute~$m$ separate local updates using~\eqref{eqn:rel-prox-grad}, and then use their average as $x_{t+1}$.
For quadratic losses, they obtain the communication complexity
$\widetilde{O}((\kappa_\ell^2/n)\log(1/\epsilon))$, which is roughly
$O(\kappa_{F/\phi}^2\log(1/\epsilon))$ in our notation, which is much worse than their 
result without averaging of $O(\kappa_{F/\phi}\log(1/\epsilon))$ given in Section~\ref{sec:stat-precond}. We further improve this to $O(\sqrt{\kappa_{F/\phi}}\log(1/\epsilon))$ using acceleration.

\citet{zhang2015disco} proposed DiSCO, an inexact damped Newton method, 
where the Newton steps are computed by a distributed conjugate gradient method with a similar preconditioner as~\eqref{eqn:reference-def}.
They obtain a communication complexity of
$\widetilde{O}((\sqrt{\kappa_\ell}/n^{1/4})\log(1/\epsilon))$
for quadratic losses and 
$\widetilde{O}(\sqrt{\kappa_\ell}(d/n)^{1/4}\log(1/\epsilon))$
for self-concordant losses. 
Comparing with~\eqref{eqn:rel-cond-quadratic} 
and~\eqref{eqn:rel-cond-non-quadratic}, in both cases they correspond to 
$O(\sqrt{\kappa_{F/\phi}}\log(1/\epsilon))$ in our notation.
\citet{reddi2016aide} use the Catalyst framework \citep{lin2015universal} to
accelerate DANE; their method, called AIDE, 
achieves the same improved complexity for quadratic functions. We obtain similar results for smooth convex functions using direct acceleration.

\citet{yuan2019convergence} revisited the analysis of DANE and found that the worse complexity of $\widetilde{O}((\kappa_\ell^2/n)\log(1/\epsilon))$ is due to the lost statistical efficiency when averaging~$m$ different updates computed by~\eqref{eqn:rel-prox-grad}.
They propose to use a single local preconditioner at the server and obtain a communication complexity of $\widetilde{O}((1+\kappa_\ell/\sqrt{n})\log(1/\epsilon))$ for quadratic functions. In addition, they propose a variant of DANE with heavy-ball momentum (DANE-HB),
and show that it has communication complexity 
$\widetilde{O}((\sqrt{\kappa_\ell}/n^{1/4})\log(1/\epsilon))$
for quadratic loss functions, matching that of DiSCO and AIDE.
For non-quadratic functions, they show DANE-HB has accelerated local convergence rate near the solution.

\citet{wang2018giant} proposed GIANT, an approximate Newton method that 
approximates the overall Hessian by the harmonic mean of the local Hessians.
It is equivalent to DANE in the quadratic case.
They obtain a communication complexity that has logarithmic dependence on the condition number but requires local sample size $n>d$.
\citet{MahajanKeerthi17} proposed a distributed algorithm based on local function approximation, which is related to the preconditioning idea of DANE.
\citet{wang2019utilizing} apply statistical preconditioning to speed up a mini-batch variant of SVRG \citep{JohnsonZhang13}, but they rely on generic Catalyst acceleration and their convergence results only hold for a very small ball around the optimum.

Distributed optimization methods that use dual variables to coordinate solutions to local subproblems include ADMM \citep{Boyd10ADMM} and 
CoCoA \citep{CoCoA2014NIPS,CoCoA2015ICML,CoCoA2017arbitrary}.
Numerical experiments demonstrate that they benefit from statistical similarities of local functions in the early iterations \citep{dscovr2019}, but their established communication complexity is no better than $O(\kappa_F\log(1/\epsilon))$.

\section{The SPAG Algorithm}
\label{sec:spag}
Although our main motivation in this paper is distributed optimization, 
the SPAG algorithm works in the general setting of minimizing relatively smooth and strongly convex functions.
In this section, we first present SPAG in the more general setting (Algorithm~\ref{algo:spag}),
then explain how to run it for distributed empirical risk minimization.

In the general setting, we consider convex optimization problems of the form~\eqref{eqn:reg-ERM}, where $\psi$ is a closed convex function and $F$ satisfies the following assumption.

\begin{assumption}\label{asmp:rel-smooth-sc}
    $F$ is $L_F$-smooth and $\sigma_F$-strongly convex. In addition, 
    it is $L_{F/\phi}$-smooth and $\sigma_{F/\phi}$-strongly convex with respect to a differentiable convex function~$\phi$, and $\phi$ itself is $L_\phi$-smooth and $\sigma_\phi$-strongly convex. 
\end{assumption}

Algorithm~\ref{algo:spag} requires an initial point $x_0\in\dom\psi$
and two parameters $L_{F/\phi}$ and $\sigma_{F/\phi}$.
During each iteration, Line~6 finds $a_{t+1}>0$ by solving a quadratic equation, 
then Line~7 calculates three scalars $\alpha_t$, $\beta_t$ and $\eta_t$, which are used in the later updates for the three vectors $y_t$, $v_{t+1}$ and $x_{t+1}$.
The function $V_t(\cdot)$ being minimized in Line~10 is defined as
\begin{align}
    \label{eqn:Vt-def}
    V_t(x) ~=~ \eta_t\bigl(\nabla F(y_t)^\top x + \psi(x)\bigr) + (1 - \beta_t) D_\phi(x, v_t) + \beta_t D_\phi(x, y_t).
\end{align}
The inequality that needs to be satisfied in Line~12 is
\begin{align} \label{eqn:Bregman-scaling}
D_{\phi}(x_{t+1}, y_t) \leq~ \alpha_t^2 G_t 
\Bigl((1 - \beta_t) D_{\phi}(v_{t+1}, v_t) + \beta_t D_{\phi}(v_{t+1}, y_t)\Bigr),
\end{align}
where $G_t$ is a scaling parameter depending on the properties of $D_\phi$. 
It is a more flexible version of the \emph{triangle scaling gain} introduced 
by \citet{Hanzely2018abpg}.

\begin{algorithm}[t]
\caption{SPAG$(L_{F / \phi}, \sigma_\rel, x_0)$}
\label{algo:spag}
\begin{algorithmic}[1]
\STATE $v_0 = x_0$, $A_0 = 0$, $B_0 = 1$, $G_{-1}=1$
\vspace{0.5ex}
\FOR{$t=0,1,2,\ldots$}
\vspace{0.5ex}
\STATE $G_t=\max\{1, G_{t-1}/2\}/2$
\vspace{0.5ex}
\REPEAT
\vspace{0.5ex}
\STATE $G_t \gets 2 G_t$
\vspace{0.5ex}
\STATE Find $a_{t+1}$ such that $a_{t+1}^2 L_{\rel} G_t = A_{t+1} B_{t+1}$ 
where $A_{t+1} \!=\! A_t + a_{t+1}$, $B_{t+1} \!=\! B_t + a_{t+1}\sigma_\rel$
\vspace{0.5ex}
\STATE $\alpha_t = \frac{a_{t+1}}{A_{t+1}}$, 
       ~~$\beta_t = \frac{a_{t+1}}{B_{t+1}}\sigma_{\rel} $, 
       ~~$\eta_t = \frac{a_{t+1}}{B_{t+1}}$
\vspace{0.5ex}
\STATE $y_t = \frac{1}{1 - \alpha_t \beta_t}\bigl((1 - \alpha_t)x_t + \alpha_t(1 - \beta_t) v_t\bigr)$ 
\vspace{0.5ex}
\STATE Compute $\nabla F(y_t)$ \emph{(requires communication)}
\vspace{0.5ex}
\STATE $v_{t+1} = \arg\min_{x} V_t(x)$
\vspace{0.5ex}
\STATE $x_{t+1} = (1 - \alpha_t) x_t + \alpha_t v_{t+1}$
\vspace{0.5ex}
\UNTIL{Inequality~\eqref{eqn:Bregman-scaling} is satisfied}
\vspace{0.5ex}
\ENDFOR
\end{algorithmic}
\end{algorithm}

As we will see in Theorem~\ref{thm:spag}, smaller $G_t$'s correspond to faster convergence rate. Algorithm~\ref{algo:spag} implements a \emph{gain-search} procedure to automatically find a small $G_t$.
At the beginning of each iteration, the algorithm always trys to set $G_t=G_{t-1}/2$ as long as $G_{t-1}\geq 2$ 
($G_{t-1}$ is divided by~$4$ in Line~3 since it is always multiplied by~$2$ in Line~5).
Whenever~\eqref{eqn:Bregman-scaling} is not satisfied, $G_t$ is multiplied by~$2$.
When the inequality~\eqref{eqn:Bregman-scaling} is satisfied, $G_t$ is within a factor of~$2$ from its smallest possible value.
The following lemma guarantees that the gain-search loop always terminates within a small number of steps (see proof in Appendix~\ref{app:amd}).

\begin{lemma}\label{lemma:Gt-bound}
    If Assumption~\ref{asmp:rel-smooth-sc} holds, 
    then the inequality~\eqref{eqn:Bregman-scaling} holds with 
    $G_t=\kappa_\phi=L_\phi/\sigma_\phi$.
\end{lemma}

Therefore, if $\phi = (1/2)\| \cdot \|^2$, then we can set $G_t=1$ and there is no need to check~\eqref{eqn:Bregman-scaling}. 
In general, Algorithm~\ref{algo:spag} always produces $G_t<2\kappa_\phi$ for all $t\geq 0$. Following the argument from \citet[][Lemma~4]{Nesterov13composite}, the total number of gain-searches performed up to iteration~$t$ is bounded by
\[
    2(t+1) + \log_2(G_t),
\]
which also bounds the total number of gradient evaluations.
Thus the overhead is roughly twice as if there were no gain-search.
Next we present a convergence theorem for SPAG.

\begin{theorem}
\label{thm:spag}
Suppose Assumption~\ref{asmp:rel-smooth-sc} holds. 
Then the sequences generated by SPAG satisfy for all $t\geq 0$,
\begin{align*}
    \bigl(\Phi(x_t) - \Phi(x_*)\bigr) + \sigma_\rel D_{\phi}(x_*, v_t) 
    \leq \frac{1}{A_t} D_{\phi}(x_*, v_0),
\end{align*}
where 
$
A_t = \frac{1}{4\sigma_\rel} \left(\prod_{\tau = 0}^{t-1} \left(1 + \gamma_\tau\right) -  \prod_{\tau = 0}^{t-1} \left(1 - \gamma_\tau\right)\right)^{2}
$,
and $\gamma_t = \frac{1}{2\sqrt{\kappa_\rel G_t}}$.
\end{theorem}
The proof of Theorem~\ref{thm:spag} relies on the techniques of \citet{nesterov2017efficiency}, and the details are given in Appendix~\ref{app:amd}.
We can estimate the convergence rate as follows:
\[
    \frac{1}{A_t} = O\biggl(\prod_{\tau=0}^t\biggl(1-\textstyle\frac{1}{\sqrt{\kappa_{F/\phi}G_\tau}}\biggr)\biggr)
    = O\biggl(\biggl(1-\textstyle\frac{1}{\sqrt{\kappa_{F/\phi}\widetilde{G}_t}}\biggr)^{\!\!t}\,\biggr),
\]
where $\widetilde{G}_t$ is such that 
$\widetilde{G}_t^{-1/2}=(1/t)\sum_{\tau=0}^t G_t^{-1/2}$, that is, $\widetilde{G}_t^{1/2}$ is the harmonic mean of $G_0^{1/2},\ldots,G_{t-1}^{1/2}$.
In addition, it can be shown that $A_t\geq t^2/(4L_{F/\phi}\widetilde{G}_t)$. 
Therefore, as $\sigma_{F/\phi}\to 0$, Theorem~\ref{thm:spag} gives
an accelerated sublinear rate:
\[
    \Phi(x_t)-\Phi(x_*) 
    \leq \frac{4L_{F/\phi}\widetilde{G_t}}{t^2} D_\phi(x_*,x_0).
\]
To estimate the worst case when $\sigma_{F/\phi}>0$, we replace $G_t$ by $\kappa_\phi$ to obtain the iteration complexity 
$O\bigl(\sqrt{\kappa_{F/\phi}\kappa_\phi}\log(1/\epsilon)\bigr)$.
Since $\kappa_{F/\phi}\kappa_\phi\approx \kappa_F$, 
this is roughly $O\bigl(\sqrt{\kappa_F}\log(1/\epsilon)\bigr)$,
the same as without preconditioning.
However, the next lemma shows that under a mild condition, we always have $G_t\to 1$ geometrically.

\begin{lemma}
\label{lemma:LS}
Suppose Assumption~\ref{asmp:rel-smooth-sc} holds and in addition, 
$\nabla^2 \phi$ is $M$-Lipschitz-continuous, i.e.,
for all $x,y \in \dom \psi$,
\[
    \left\| \nabla^2 \phi(x) - \nabla^2 \phi(y) \right\| \leq M \| x - y \|.
\]
Then the inequality~\eqref{eqn:Bregman-scaling} holds with 
\begin{align}\label{eq:G_t}
    G_t  = \min\bigl\{\kappa_{\phi}, ~1 + (M/\sigma_{\phi})d_t\bigr\},
\end{align}
where $d_t = \|v_{t+1} - v_t\| + \|v_{t+1} - y_t\| + \|x_{t+1} - y_t\|$.
\end{lemma}

In particular, if $\phi$ is quadratic, then we have $M=0$ and $G_t=1$ always satisfies~\eqref{eqn:Bregman-scaling}.
In this case, the convergence rate in Theorem~\ref{thm:spag} satisfies
$1/A_t=O\bigl(\bigl(1-1/\sqrt{\kappa_{F/\phi}}\bigr)^{\!t}\,\bigr)$.

In general, $M\neq 0$, but it can be shown that the sequences generated by
Algorithm~\ref{algo:spag}, $\{x_t\}$, $\{y_t\}$ and $\{v_t\}$ all converge
to $x_*$ at the rate $\bigl(1-1/\sqrt{\kappa_F}\bigr)^t$
\citep[see, e.g.,][Theorem~1]{LinXiao2015homotopy}.
As a result, $d_t\to 0$ and thus $G_t\to 1$ at the same rate.
Consequently, the convergence rate established in Theorem~\ref{thm:spag} 
quickly approaches 
$O\bigl(\bigl(1-1/\sqrt{\kappa_{F/\phi}}\bigr)^{\!t}\,\bigr)$.

\subsection{Implementation for Distributed Optimization}
\label{sec:impl-distr-erm}

In distributed optimization, Algorithm~\ref{algo:spag} is implemented at the server. During each iteration, communication between the server and the workers only happens when computing $\nabla F(y_t)$.
Checking if the inequality~\eqref{eqn:Bregman-scaling} holds locally requires that the server has access to the preconditioner~$\phi$.

If the datasets on different workers are i.i.d.~samples from the same source distribution, then we can use any $f_j$ in the definition of~$\phi$ in~\eqref{eqn:reference-def} and assign worker~$j$ as the server.
However, this is often not the case in practice and obtaining i.i.d.~datasets on different workers may involve expensive shuffling and exchanging large amount of data among the workers.
In this case, a better alternative is to randomly sample small portions of the data on each worker and send them to a dedicated server.
We call this sub-sampled dataset $\mathcal{D}_0$ and the local loss at the server $f_0$, which is defined the same way as in~\eqref{eqn:fj-def}.
Then the server implements Algorithm~\ref{algo:spag} with
$\phi(x)=f_0(x)+(\mu/2)\|x\|^2$.
Here we only need $\mathcal{D}_0$ be a uniform sub-sample of 
$\cup_{j=1}^m\mathcal{D}_j$, which is critical for effective preconditioning.
On the other hand, it is not a problem at all if the datasets at the workers,
 $\mathcal{D}_1,\ldots,\mathcal{D}_m$, are not shuffled to to be i.i.d., because it 
does not change the average gradients $\nabla F(y_t)$. 
In the rest of the paper, we omit the subscript to simply use~$f$ to represent the local empirical loss function.
As discussed in Section~\ref{sec:stat-precond}, if 
\begin{equation}\label{eqn:f0-Hessian-approx}
    \left\|\nabla^2 f(x) - \nabla^2 F(x)\right\| \leq \mu, 
   \quad \forall\, x\in\dom\psi
\end{equation}
with high probability, then according to~\eqref{eqn:rel-sigma-mu}, we can choose\[ 
    L_{F/\phi}=1, \qquad \sigma_{F/\phi} = \frac{\sigma_F}{\sigma_F+2\mu}
\]
as the input to Algorithm~\ref{algo:spag}.
In the next section, we leverage matrix concentration bounds to estimate how $\mu$ varies with the number of subsamples~$n$.
With sufficiently large~$n$, we can make $\mu$ small so that the relative condition number $\kappa_{F/\phi}=1+2\mu/\sigma_F$ is much smaller than $\kappa_F$.

\section{Bounding the Relative Condition Number}
\label{sec:concentration}

In this section, we derive refined matrix concentration bounds for linear prediction models.
Suppose the overall dataset consists of~$N$ samples $\{z_1,\ldots,z_N\}$, where each $z_i=(a_i,b_i)$ with $a_i\in\R^d$ being a feature vector and $b_i$ the corresponding label or regression target. 
Linear models (including logistic and ridge regression) have the form
$\ell(x, z_i) = \ell_i(a_i^\top x) + \frac{\lambda}{2}\|x\|^2$, 
where $\ell_i$ is twice differentiable and may depend on $b_i$, and $\lambda>0$. We further assume that $\ell^\second_i = \ell^\second_j$ for all $i$ and $j$, which is valid for logistic and ridge regression as well. Since $f(x) = (1/n)\sum_{i=1}^n\ell(x,z_i)$, we have
\begin{equation}\label{eqn:linear-hessian}
    \nabla^2 f(x) = \frac{1}{n}\sum_{i=1}^n \ell_i^{\prime\prime}(a_i^\top x) a_i a_i^\top + \lambda I_d .
\end{equation}
Here we omit the subscript~$j$ in $f_j$ since we only need one subsampled dataset at the server, as explained in Section~\ref{sec:impl-distr-erm}.
For the overall loss function defined in~\eqref{eqn:average-loss}, the Hessian $\nabla^2 F(x)$ is defined similarly by replacing~$n$ with~$N$.

We assume for simplicity that the strong convexity of~$F$ mainly comes from regularization, that is, $\sigma_F = \sigma_\ell = \lambda$, but the results can be easily extended to account for the strong convexity from data.
We start by showing tight results for quadratics, and then provide uniform concentration bounds of Hessians for more general loss functions. Finally, we give a refined bound when the~$a_i$'s are sub-Gaussian.

\subsection{Quadratic Case}
We assume in this section that  $\ell_i(a_i^\top x )=(a_i^\top x-b_i)^2/2$, 
and  that there exists a constant $R$ such that $\|a_i\|\leq R$ for all $i=1,\ldots,N$. In this case we have $L_\ell=R^2$ and $\kappa_\ell=R^2/\lambda$.
Since the Hessians do not depend on $x$, we use the notation
\[
    H_F=\nabla^2 F(x), \quad H_f=\nabla^2 f(x).
\]
Previous works 
\citep{shamir2014communication, reddi2016aide, yuan2019convergence} 
use the Hoeffding bound~\eqref{eqn:matrix-hoeffding} to obtain
\begin{equation}\label{eqn:rel-by-additive}
    \left(1+\frac{2\mu}{\lambda}\right)^{\!-1} (H_f+\mu I_d) \preceq H_F
    \preceq H_f + \mu I_d,
\end{equation}
\begin{equation}\label{eqn:mu-hoeffding}
 \text{ with \ \ \ \ \ \ }  \mu  = \frac{R^2}{\sqrt{n}}\sqrt{32\log(d/\delta)} .
\end{equation} 
Our result is given in the following theorem.
\begin{theorem}
\label{thm:concentration_quadratics}
Suppose $\ell_i$ is quadratic and $\|a_i\|\leq R$ for all~$i$.
For a fixed $\delta>0$, if $n > \frac{28}{3}\log\left(\frac{2d}{\delta}\right)$, then the following inequality holds with probability at least $1 - \delta$:
\begin{equation}\label{eqn:rel-by-mult}\!\!\!
\left(\frac{3}{2} + \frac{2 \mu}{\lambda}\right)^{\!-1} \!\!\left(H_f + \mu I_d \right) \preceq H_F \preceq 2 \left(H_f + \mu I_d \right),
\end{equation}
\begin{equation} \label{eqn:quadratic-mu}
\text{ with \ \ \ \ \ \ } \mu = \frac{1}{2}\left(\frac{28 R^2}{3n}\log\left(\frac{2d}{\delta}\right) - \lambda \right)^+.
\end{equation}
Thus, for this choice of $\mu$, $\sigma_\rel = \left(\frac{3}{2} + \frac{2 \mu}{\lambda}\right)^{-1}$, $L_\rel = 2$ and so $\kappa_\rel = O\left(1 + \frac{\kappa_\ell}{n}\log\left(\frac{d}{\delta}\right)\right)$ with probability $1 - \delta$.
\end{theorem}

Theorem~\ref{thm:concentration_quadratics} improves on the result in~\eqref{eqn:mu-hoeffding} by a factor of $\sqrt{n}$. The reason is that matrix inequality~\eqref{eqn:rel-by-additive} is derived from the additive bound $\|H_f-H_F\|\leq \mu$ \citep[e.g.,][]{shamir2014communication, yuan2019convergence}.
We derive the matrix inequality~\eqref{eqn:rel-by-mult} directly from a multiplicative bound using the matrix Bernstein inequality (see proof in Appendix~\ref{app:concentration_quad}).
Note that by using matrix Bernstein instead of matrix Hoeffding inequality~\citep{tropp2015introduction}, one can refine the bound for~$\mu$ in~\eqref{eqn:rel-by-additive} from $L_\ell/ \sqrt{n}$ to $\sqrt{L_\ell L_F / n}$, which can be as small as $L_\ell / n$ in the extreme case when all the $a_i$'s are orthogonal.
Our bound in~\eqref{eqn:quadratic-mu} states that $\mu=\widetilde{O}(L_\ell/n)$ in general for quadratic problems, leading to $\kappa_{F/\phi}=\widetilde{O}(1+\kappa_\ell/n)$.

\begin{remark}
\emph{
Theorem~\ref{thm:concentration_quadratics} is proved by assuming random sampling with replacement. In practice, we mostly use random sampling without replacement, which usually concentrates even more than with replacement~\citep{hoeffding1963probability}.
}
\end{remark}

\begin{remark}\label{remark:higher_mu}
\emph{
In terms of reducing $\kappa_{F/\phi}$, there is not much benefit to having $\mu<\lambda$. Indeed, higher values of $\mu$ regularize the inner problem of minimizing $V_t(x)$ in~\eqref{eqn:Vt-def}, because the condition number of $D_\phi(x,y)=D_f(x,y)+(\mu/2)\|x-y\|^2$ is $(L_f + \mu) / (\lambda + \mu)$. Increasing $\mu$ can thus lead to substantially easier subproblems when $\mu > \lambda$, which reduces the computation cost at the server, although this may sometimes affect the rate of convergence.
}
\end{remark}

\subsection{Non-quadratic Case}

For non-quadratic loss functions, we need $\nabla^2 f (x)$ to be a good approximation of $\nabla^2 F(x)$ for all iterations of the SPAG algorithm. It is tempting to argue that concentration only needs to hold for the iterates of SPAG, and a union bound would then give an extra $\log T$ factors for $T$ iterations. Yet this only works for one step since $x_t$ depends on the points chosen to build $f$ for $t > 0$, so the $\ell^\second(a_i^\top x_t)a_i a_i^\top$ are not independent for different $i$ (because of $x_t$). Therefore, the concentration bounds need to be written at points that do not depend on $f$. In order to achieve this, we restrict the optimization variable within a bounded convex set and prove uniform concentration of Hessians over the set.
Without loss of generality, we consider optimization problems constrained 
in $\cB(0, D)$, the ball of radius $D$ centered at 0.
Correspondingly, we set the nonsmooth regularization function as
$\psi(x)=0$ if $x\in\cB(0,D)$ and infinity otherwise.

If the radius~$D$ is small, it is then possible to leverage the quadratic bound by using the inequality
\begin{align*}
    &\|H_f(x) - H_F(x)\| ~\leq~  \|H_f(x) - H_f(y)\| \\
    &\qquad\qquad + \|H_f(y) - H_F(y)\| + \|H_F(x) - H_F(y)\|.
\end{align*}
Thus, under a Lipschitz-continuous Hessian assumption (which we have), only concentration at point $y$ matters. Yet, such bounding is only meaningful when $x$ is close to $y$, thus leading to the very small convergence radius of \citet[Theorem~13]{wang2019utilizing}, in which they use concentration at the optimal point $x_*$. Using this argument for several $y$'s that pave $\cB(0,D)$ leads to an extra $\sqrt{d}$ multiplicative factor since concentration needs to hold at exponentially (in $d$) many points, as discussed in Section~\ref{sec:stat-precond}. We take a different approach in this work, and proceed by directly bounding the supremum for all $x \in \cB(0,D)$, thus looking for the smallest $\mu$ that satisfies: 
\begin{equation}\label{eq:sup_bound}
  \sup_{x \in \mathcal{B}(0, D)} \| H_f(x) - H_F(x)  \|_{\rm op} \leq \mu.
\end{equation}
Equation~\eqref{eqn:rel-by-additive} can then be used with this specific $\mu$. We now introduce Assumption~\ref{assumption:sudakov}, which is for example verified for logistic regression with $B_\ell = 1/4$ and $M_\ell = 1$.

\begin{assumption}
\label{assumption:sudakov}
There exist $B_\ell$ and $M_\ell$ such that $\ell_i^\second$ is $M_\ell$-Lipschitz continuous and $0 \leq \ell_i^\second(a^\top x) \leq B_\ell$ almost surely for all $x \in \mathcal{B}(0, D)$.
\end{assumption}

\begin{theorem}
\label{thm:concentration_sudakov}
If $\ell_i$ satisfies Assumption~\ref{assumption:sudakov}, then Equation~\eqref{eq:sup_bound} is satisfied with probability at least $1 - \delta$ for 
$$ \mu = \sqrt{4 \pi} \frac{R^2}{\sqrt{n}}\left( B_\ell \left[2 + \sqrt{\frac{1}{2\pi}\log(\delta^{-1})}\right] + R M_\ell D\right).
$$
\end{theorem}

\begin{proof}[Sketch of proof]
The high probability bound on the supremum is obtained using Mc Diarmid inequality~\citep{boucheron2013concentration}. This requires a bound on its expectation, which is obtained using symmetrization and the Sudakov-Fernique Lemma~\citep{boucheron2013concentration}. The complete proof can be found in Appendix~\ref{app:concentration_gen}.
\end{proof}

The bound of Theorem~\ref{thm:concentration_sudakov} is relatively tight as long as $RM_\ell D < B_\ell \sqrt{\log(\delta^{-1})}$. Indeed, using the matrix Bernstein inequality for a fixed $x \in \cB(0, D)$ would yield $\mu = O\left(R \sqrt{L_F} B_\ell \log(d / \delta) / \sqrt{n}\right)$. Therefore, Theorem~\ref{thm:concentration_sudakov} is tight up to a factor $R / \sqrt{L_F}$ in this case.

\subsection{Sub-Gaussian Bound}
We show in this section that the bound of Theorem~\ref{thm:concentration_sudakov} can be improved under a stronger sub-Gaussian assumption on $a$.

\begin{definition}
\label{def:subgaussian}
The random variable $a \in \R^d$ is sub-Gaussian with parameter $\rho>0$ if one has for all $\epsilon > 0$, $x \in \cB(0, D)$:
\begin{equation}\label{eq:subgauss}
\dP(|a_i^\top x)|\ge \rho \epsilon)\le 2 e^{-\frac{\epsilon^2}{2\|x\|^2}}.
\end{equation}
\end{definition}

\begin{theorem}
\label{thm:concentration_subgaussian}
If $\ell_i$ satisfies Assumption~\ref{assumption:sudakov} and the $a_i$ are sub-Gaussian with constant $\rho$, then denoting $\tilde{B} = B_\ell / (M_\ell D)$, there exists $C > 0$ such that Equation~\eqref{eq:sup_bound} is satisfied with probability $1 - \delta$ for
\begin{equation*}
    \mu = C \frac{\rho^2 M_\ell D}{\sqrt{n}} (d + \log(\delta^{-1}))\left[\frac{\rho + \tilde{B}}{\sqrt{d}} + \frac{\rho + (R^2 \tilde{B})^{\frac{1}{3}}}{\sqrt{n}}\right].
\end{equation*}
\end{theorem}

\begin{proof}[Sketch of proof]
This bound is a specific instantiation of a more general result based on chaining, which is a standard argument for proving results on suprema of empirical processes~\citep{boucheron2013concentration}. The complete proof can be found in Appendix~\ref{app:concentration_subgaussian}.
\end{proof}

The sub-Gaussian assumption \eqref{eq:subgauss} always holds with $\rho=R$, the almost sure bound on $\|a_i\|$. However Theorem~\ref{thm:concentration_subgaussian} improves over Theorem~\ref{thm:concentration_sudakov} only with a stronger sub-Gaussian assumption, i.e., when $\rho<R$. In particular for $a_i$ uniform over $\cB(0,R)$, one has $\rho=R/\sqrt{d}$. Assuming further that the $(R^2B)^{1/3} / \sqrt{n}$ term dominates yields $\mu=O(R^2 (R^2 B)^{1/3}/n)$, a $\sqrt{n}$ improvement over Theorem~\ref{thm:concentration_sudakov}. We expect tighter versions of Theorem~\ref{thm:concentration_subgaussian}, involving the effective dimension $d_{\rm eff}$ of vectors $a_i$ instead of the full dimension $d$, to hold. 

\section{Experiments}
\label{sec:experiments}

We have seen in the previous section that preconditioned gradient methods can outperform gradient descent by a large margin in terms of communication rounds, which was already observed empirically~\citep{shamir2014communication, reddi2016aide, yuan2019convergence}. We compare in this section the performances of SPAG with those of DANE and its heavy-ball acceleration, HB-DANE~\citep{yuan2019convergence}, as well as accelerated gradient descent (AGD). For this, we use two datasets from LibSVM\footnote{Accessible at \url{https://www.csie.ntu.edu.tw/~cjlin/libsvmtools/datasets/binary.html}}, RCV1~\cite{lewis2004rcv1} and the preprocessed version of KDD2010 (algebra)~\cite{yu2010feature}. Due to its better convergence guarantees~\citep{shamir2014communication, yuan2019convergence}, DANE refers in this section to the proximal gradient method with the Bregman divergence associated to $\phi = f_1 + (\mu / 2)\| \cdot \|^2$ (without averaging over $m$ workers). We use SPAG with $\sigma_\rel = 1 / (1 + 2\mu / \lambda)$ and HB-DANE with $\beta = (1 - (1 + 2\mu / \lambda)^{-1/2})^2$. Fine tuning these parameters only leads to comparable small improvements for both algorithms. We tune both the learning rate and the momentum of AGD. 

Note that, as mentioned in Section~\ref{sec:impl-distr-erm}, the number of nodes used by SPAG does not affect its iteration complexity (but change the parallelism of computing $\nabla F(x_t)$). Only the size $n$ of the dataset used for preconditioning matters. We initialize all algorithms at the same point, which is the minimizer of the server's entire local loss (regardless of how many samples are used for preconditioning). 

\textbf{Tuning $\mu$.} Although $\mu$ can be estimated using concentrations results, as done in Section~\ref{sec:concentration}, these bounds are too loose to be used in practice. Yet, they show that $\mu$ depends very weakly on $\lambda$. This is verified experimentally, and we therefore use the same value for $\mu$ regardless of $\lambda$. To test the impact of $\mu$ on the iteration complexity, we fix a step-size of~$1$ and plot the convergence speed of SPAG for several values of $\mu$. We see on Figure~\ref{fig:mu_comparison} that the value of $\mu$ drastically affects convergence, actually playing a role similar to the inverse of a step-size. Indeed, the smaller the $\mu$ the faster the convergence, up to a point at which the algorithm is not stable anymore. Convergence could be obtained for smaller values of $\mu$ by taking a smaller step-size. Yet, the step-size needs to be tuned for each value of $\mu$, and we observed that this does not lead to significant improvements in practice. Thus, we stick to the guidelines for DANE by~\citet{shamir2014communication}, i.e., we choose $L_\rel = 1$ and tune $\mu$.

\textbf{Line search for $G_t$.}
As explained in Section~\ref{sec:spag}, the optimal $G_t$ is obtained through a line search. Yet, we observed in all our experiments that $G_t = 1$ most of the time. This is due to the fact that 
we start at the minimizer of the local cost function, which can be close to the global solution.
In addition, Equation~\eqref{eq:G_t} can actually be verified for $G_t < 1$, even in the quadratic. Therefore, the line search generally has no added cost (apart from checking that $G_t = 1$ works) and the effective rate in our experiments is $\kappa_\rel^{-1/2}$. Experiments displayed in Figure~\ref{fig:allfig} use $G_t = 1$ for simplicity.

\begin{figure}[t]
\begin{subfigure}[normal]{0.5\linewidth}
    \includegraphics[width=\linewidth]{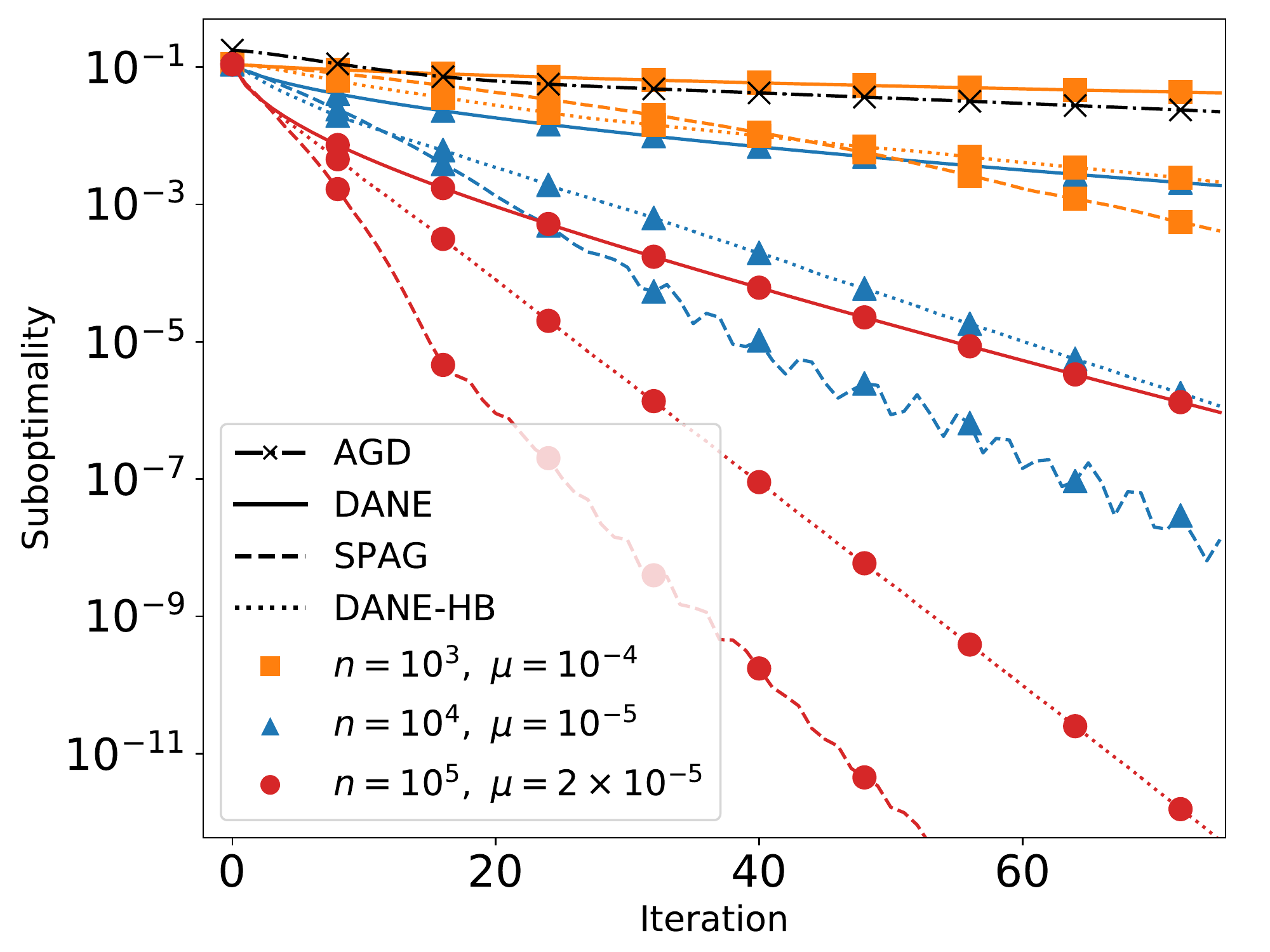}
    \caption{RCV1, $\lambda = 10^{-7}$}
    \label{fig:exp_rcv1_low_reg}
\end{subfigure}
\begin{subfigure}[normal]{0.5\linewidth}
    \includegraphics[width=\linewidth]{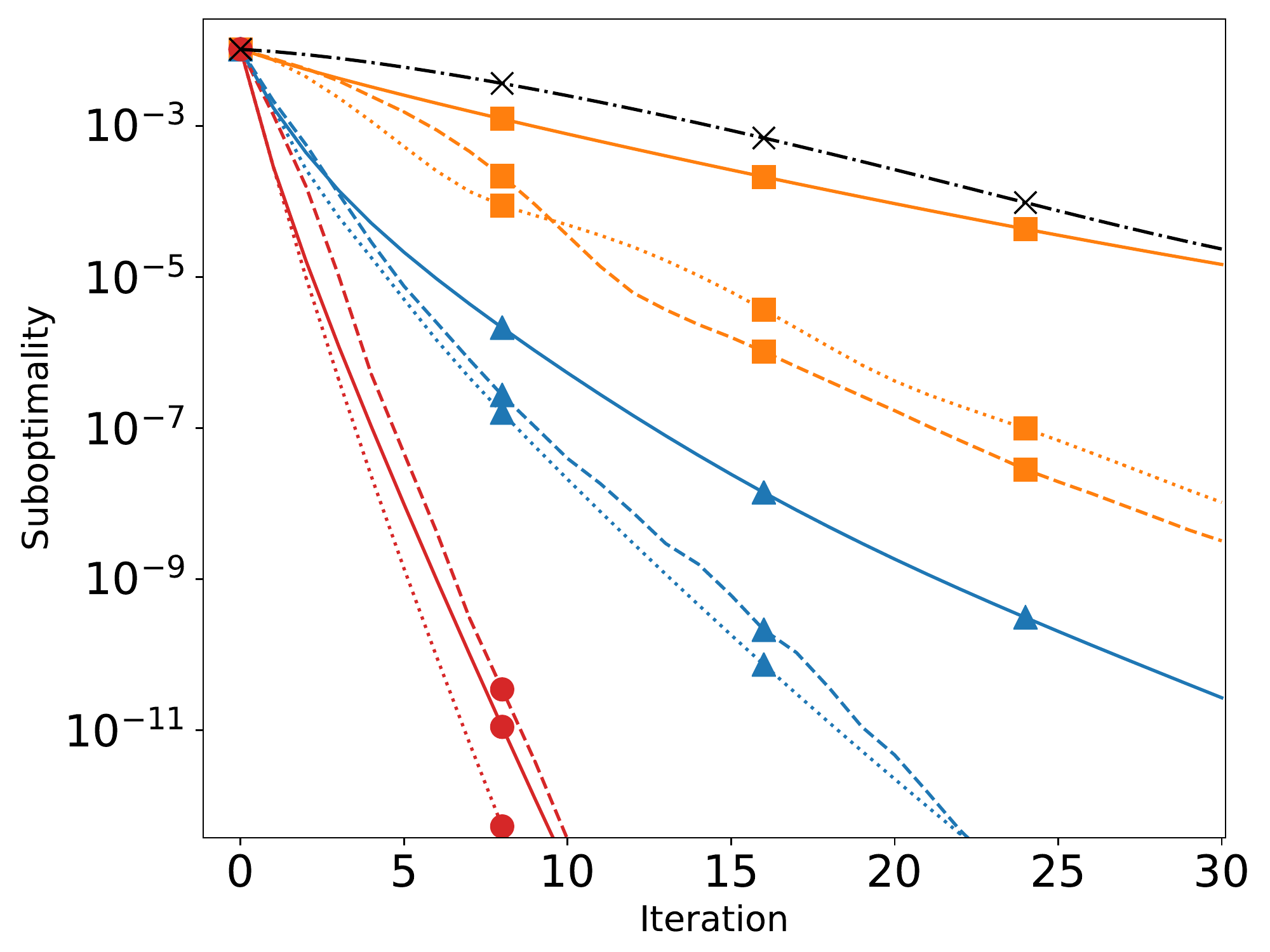}
    \caption{RCV1, $\lambda = 10^{-5}$}
    \label{fig:exp_rcv1_high_reg}
\end{subfigure}\\
\begin{subfigure}[normal]{0.5\linewidth}
    \includegraphics[width=\linewidth]{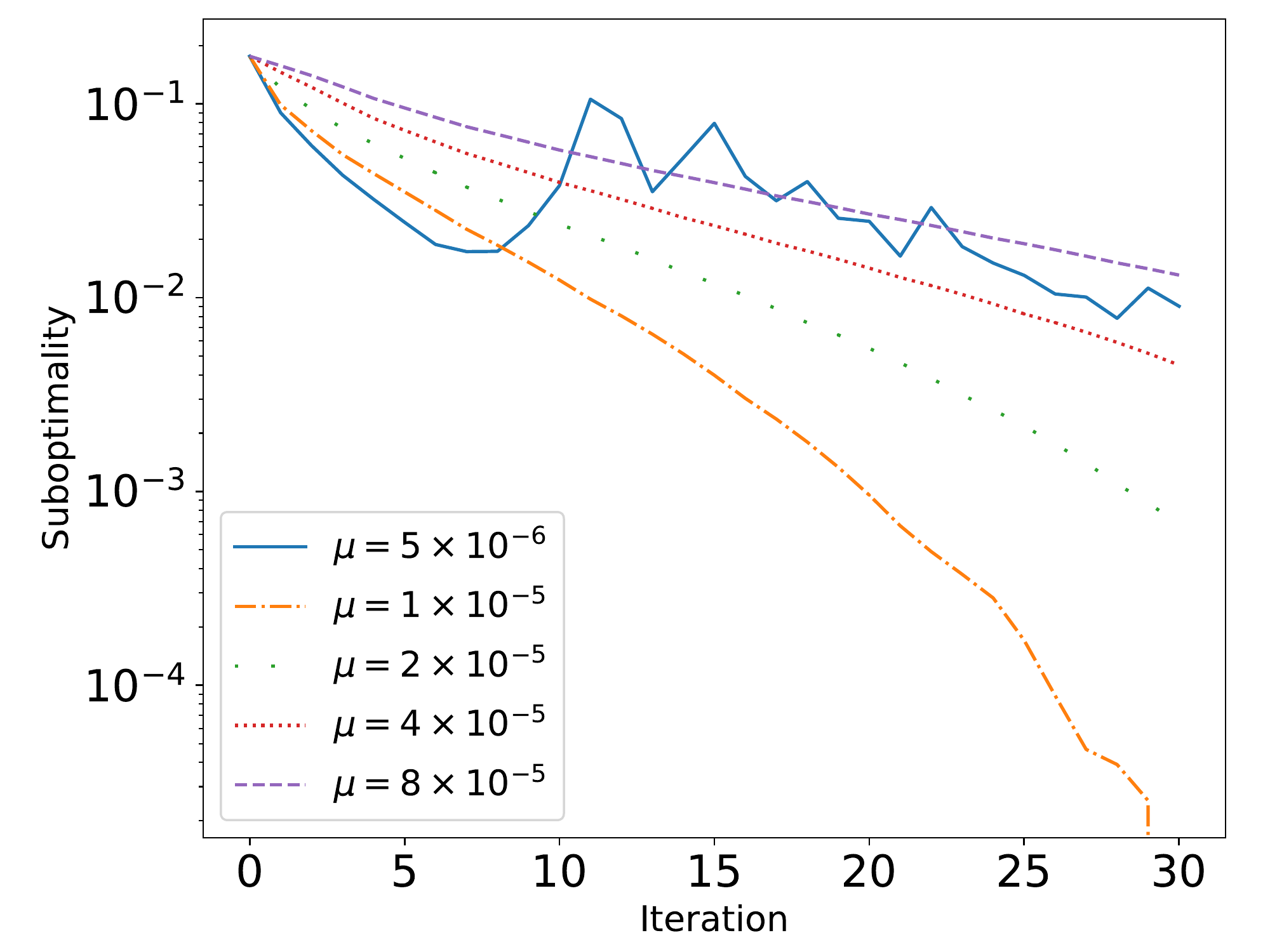}
    \caption{Effect of $\mu$ on the convergence speed of SPAG\\ on RCV1 with $\lambda = 10^{-7}$ and $n=10^4$.}
    \label{fig:mu_comparison}
\end{subfigure}
\begin{subfigure}[normal]{0.5\linewidth}
    \includegraphics[width=\linewidth]{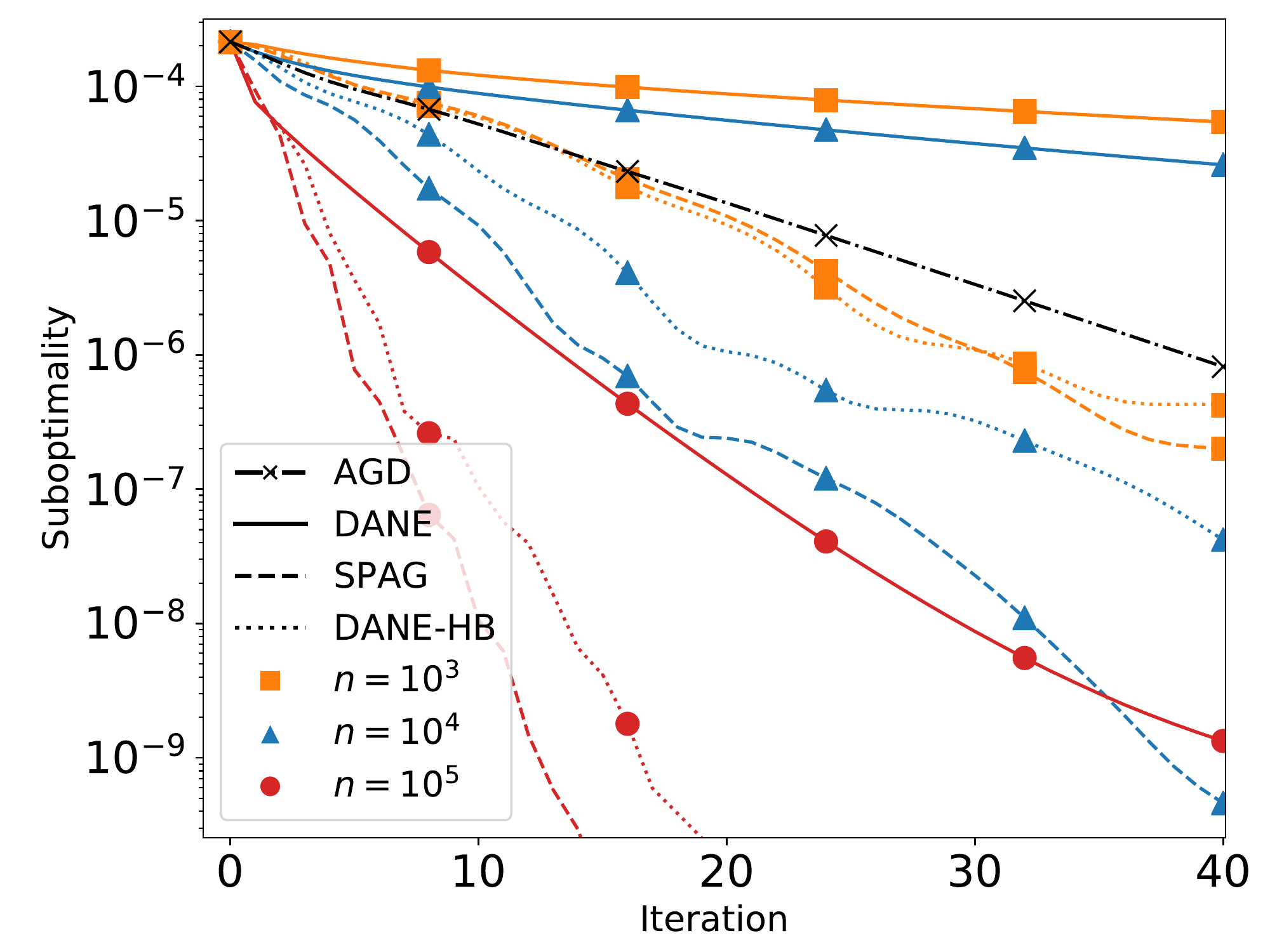}
    \caption{KDD2010, $\lambda=10^{-7}$. We use $\mu = 0.1 / (2n)$, except for $n=10^3$, where $\mu=10^{-5}$, and $L_\rel = 2$ for SPAG.}
    \label{fig:exp_kdd2010}
\end{subfigure}
\caption{Experimental results} \label{fig:allfig}
\end{figure}

\textbf{RCV1.}
Figures~\ref{fig:exp_rcv1_high_reg} and~\ref{fig:exp_rcv1_low_reg} present results for the RCV1 dataset with different regularizations. All algorithms are run with $N = 677 399$ (split over 4 nodes) and $d=47236$. We see that in Figure~\ref{fig:exp_rcv1_high_reg}, the curves can be clustered by values of~$n$, meaning that when regularization is relatively high ($\lambda=10^{-5}$), increasing the preconditioning sample size has a greater effect than acceleration since the problem is already well-conditioned. In particular, acceleration does not improve the convergence rate when $n=10^5$ and $\lambda=10^{-5}$. When regularization is smaller ($\lambda = 10^{-7}$), SPAG and HB-DANE outperform DANE even when ten times less samples are used for preconditioning, as shown in Figure~\ref{fig:exp_rcv1_low_reg}. Finer tuning (without using the theoretical parameters) of the momentum marginally improves the performances of SPAG and HB-DANE, at the cost of a grid search. 
SPAG generally outperforms HB-DANE in our experiments, but both methods have comparable asymptotic rates.

\textbf{KDD2010.} Figure~\ref{fig:exp_kdd2010} presents the results of larger scale experiments on a random subset of the KDD2010 dataset with $N=7557074$ (split over 80 nodes), $d=20216830$ and $\lambda = 10^{-7}$. The conclusions are similar to the experiments on RCV1,  i.e., acceleration allows to use significantly less samples at the server for a given convergence speed. AGD competes with DANE when $\lambda$ and $n$ are small, but it is outperformed by SPAG in all our experiments. More experiments investigating the impact of line search, tuning and inaccurate local solutions are presented in Appendix~\ref{app:experiments}. 

\section{Conclusion}
We have introduced SPAG, an accelerated algorithm that performs statistical preconditioning for large-scale optimization. 
Although our motivation in this paper is for distributed empirical risk minimization, SPAG applies to much more general settings.
We have given tight bounds on the relative condition number, a crucial quantity to understand the convergence rate of preconditioned algorithms. We have also shown, both in theory and in experiments, that acceleration allows SPAG to efficiently leverage rough preconditioning with limited number of local samples. Preliminary experiments suggest that SPAG is more robust to inaccurate solution of the inner problems than HB-DANE. Characterizing the effects of inaccurate inner solutions in the preconditioning setting would be an interesting extension of this work.

\section{Acknowledgements}

This work was funded in part by the French government under management of Agence Nationale de la Recherche as part of the ``Investissements d'avenir'' program, reference ANR-19-P3IA-0001 (PRAIRIE 3IA Institute). We also acknowledge support from the European Research Council (grant SEQUOIA 724063) and from the MSR-INRIA joint centre. 

\bibliographystyle{plainnat}
\bibliography{biblio}

\begin{thebibliography}{36}
\providecommand{\natexlab}[1]{#1}
\providecommand{\url}[1]{\texttt{#1}}
\expandafter\ifx\csname urlstyle\endcsname\relax
  \providecommand{\doi}[1]{doi: #1}\else
  \providecommand{\doi}{doi: \begingroup \urlstyle{rm}\Url}\fi

\bibitem[Arjevani and Shamir(2015)]{ArjevaniShamir2015NIPS}
Yossi Arjevani and Ohad Shamir.
\newblock Communication complexity of distributed convex learning and
  optimization.
\newblock In \emph{Advances in Neural Information Processing Systems 28}, pages
  1756--1764, 2015.

\bibitem[Bauschke et~al.(2017)Bauschke, Bolte, and
  Teboulle]{bauschke2017descent}
Heinz~H. Bauschke, J{\'e}r{\^o}me Bolte, and Marc Teboulle.
\newblock A descent lemma beyond {L}ipschitz gradient continuity: first-order
  methods revisited and applications.
\newblock \emph{Mathematics of Operations Research}, 42\penalty0 (2):\penalty0
  330--348, 2017.

\bibitem[Beck(2017)]{Beck17book}
Amir Beck.
\newblock \emph{First-Order Methods in Optimization}.
\newblock MOS-SIAM Series on Optimization. SIAM, 2017.

\bibitem[Beck and Teboulle(2009)]{BeckTeboulle09fista}
Amir Beck and Marc Teboulle.
\newblock A fast iterative shrinkage-thresholding algorithm for linear inverse
  problems.
\newblock \emph{SIAM Journal on Imaging Sciences}, 2\penalty0 (1):\penalty0
  183--202, 2009.

\bibitem[Boucheron et~al.(2013)Boucheron, Lugosi, and
  Massart]{boucheron2013concentration}
St{\'e}phane Boucheron, G{\'a}bor Lugosi, and Pascal Massart.
\newblock \emph{Concentration Inequalities: A Nonasymptotic Theory of
  Independence}.
\newblock Oxford University Press, 2013.

\bibitem[Boyd et~al.(2010)Boyd, Parikh, Chu, Peleato, and Eckstein]{Boyd10ADMM}
Stephen Boyd, Neal Parikh, Eric Chu, Borja Peleato, and Jonathan Eckstein.
\newblock Distributed optimization and statistical learning via the alternating
  direction method of multipliers.
\newblock \emph{Foundations and Trends in Machine Learning}, 3\penalty0
  (1):\penalty0 1--122, 2010.

\bibitem[Chen and Teboulle(1993)]{ChenTeboulle93}
Gong Chen and Marc Teboulle.
\newblock Convergence analysis of a proximal-like minimization algorithm using
  {Bregman} functions.
\newblock \emph{SIAM Journal on Optimization}, 3\penalty0 (3):\penalty0
  538--543, August 1993.

\bibitem[Dragomir et~al.(2019)Dragomir, Taylor, d'Aspremont, and
  Bolte]{dragomir2019optimal}
Radu-Alexandru Dragomir, Adrien Taylor, Alexandre d'Aspremont, and
  J{\'e}r{\^o}me Bolte.
\newblock Optimal complexity and certification of {B}regman first-order
  methods.
\newblock \emph{arXiv preprint arXiv:1911.08510}, 2019.

\bibitem[Hanzely et~al.(2018)Hanzely, Richtarik, and Xiao]{Hanzely2018abpg}
Filip Hanzely, Peter Richtarik, and Lin Xiao.
\newblock Accelerated {B}regman proximal gradient methods for relatively smooth
  convex optimization.
\newblock arXiv:1808.03045, 2018.

\bibitem[Hoeffding(1963)]{hoeffding1963probability}
Wassily Hoeffding.
\newblock Probability inequalities for sums of bounded random variables.
\newblock \emph{Journal of the American Statistical Association}, 58\penalty0
  (301):\penalty0 13--30, 1963.

\bibitem[Jaggi et~al.(2014)Jaggi, Smith, Takac, Terhorst, Krishnan, Hofmann,
  and Jordan]{CoCoA2014NIPS}
Martin Jaggi, Virginia Smith, Martin Takac, Jonathan Terhorst, Sanjay Krishnan,
  Thomas Hofmann, and Michael~I. Jordan.
\newblock Communication-efficient distributed dual coordinate ascent.
\newblock In \emph{Advances in Neural Information Processing Systems 27}, pages
  3068--3076, 2014.

\bibitem[Johnson and Zhang(2013)]{JohnsonZhang13}
Rie Johnson and Tong Zhang.
\newblock Accelerating stochastic gradient descent using predictive variance
  reduction.
\newblock In \emph{Advances in Neural Information Processing Systems 26}, pages
  315--323, 2013.

\bibitem[Lewis et~al.(2004)Lewis, Yang, Rose, and Li]{lewis2004rcv1}
David~D Lewis, Yiming Yang, Tony~G Rose, and Fan Li.
\newblock {RCV1}: A new benchmark collection for text categorization research.
\newblock \emph{Journal of machine learning research}, 5\penalty0
  (Apr):\penalty0 361--397, 2004.

\bibitem[Lin et~al.(2015)Lin, Mairal, and Harchaoui]{lin2015universal}
Hongzhou Lin, Julien Mairal, and Zaid Harchaoui.
\newblock A universal catalyst for first-order optimization.
\newblock In \emph{Advances in Neural Information Processing Systems}, pages
  3384--3392, 2015.

\bibitem[Lin and Xiao(2015)]{LinXiao2015homotopy}
Qihang Lin and Lin Xiao.
\newblock An adaptive accelerated proximal gradient method and its homotopy
  continuation for sparse optimization.
\newblock \emph{Computational Optimization and Applications}, 60\penalty0
  (3):\penalty0 633--674, Apr 2015.

\bibitem[Lu et~al.(2018)Lu, Freund, and Nesterov]{lu2018relatively}
Haihao Lu, Robert~M Freund, and Yurii Nesterov.
\newblock Relatively smooth convex optimization by first-order methods, and
  applications.
\newblock \emph{SIAM Journal on Optimization}, 28\penalty0 (1):\penalty0
  333--354, 2018.

\bibitem[Ma et~al.(2015)Ma, Smith, Jaggi, Jordan, Richt\'{a}rik, and
  Tak\'{a}\v{c}]{CoCoA2015ICML}
Chenxin Ma, Virginia Smith, Martin Jaggi, Michael~I. Jordan, Peter
  Richt\'{a}rik, and Martin Tak\'{a}\v{c}.
\newblock Adding vs. averaging in distributed primal-dual optimization.
\newblock In \emph{Proceedings of the International Conference on Machine
  Learning}, pages 1973--1982, 2015.

\bibitem[Ma et~al.(2017)Ma, Smith, Jaggi, Jordan, Richt\'{a}rik, and
  Tak\'{a}\v{c}]{CoCoA2017arbitrary}
Chenxin Ma, Virginia Smith, Martin Jaggi, Michael~I. Jordan, Peter
  Richt\'{a}rik, and Martin Tak\'{a}\v{c}.
\newblock Distributed optimization with arbitrary local solvers.
\newblock \emph{Optimization Methods and Software}, 32\penalty0 (4):\penalty0
  813--848, 2017.

\bibitem[Mahajan et~al.(2018)Mahajan, Agrawal, Keerthi, Sellamanickam, and
  Bottou]{MahajanKeerthi17}
Dhruv Mahajan, Nikunj Agrawal, S.~Sathiya Keerthi, Sundararajan Sellamanickam,
  and Leon Bottou.
\newblock An efficient distributed learning algorithm based on effective local
  functional approximations.
\newblock \emph{Journal of Machine Learning Research}, 19\penalty0
  (74):\penalty0 1--37, 2018.

\bibitem[Nesterov(2004)]{Nesterov04book}
Yurii Nesterov.
\newblock \emph{Introductory Lectures on Convex Optimization: A Basic Course}.
\newblock Kluwer, Boston, 2004.

\bibitem[Nesterov(2013)]{Nesterov13composite}
Yurii Nesterov.
\newblock Gradient methods for minimizing composite functions.
\newblock \emph{Mathematical Programming, Ser.\ B}, 140:\penalty0 125--161,
  2013.

\bibitem[Nesterov and Stich(2017)]{nesterov2017efficiency}
Yurii Nesterov and Sebastian~U Stich.
\newblock Efficiency of the accelerated coordinate descent method on structured
  optimization problems.
\newblock \emph{SIAM Journal on Optimization}, 27\penalty0 (1):\penalty0
  110--123, 2017.

\bibitem[Reddi et~al.(2016)Reddi, Kone{\v{c}}n{\`y}, Richt{\'a}rik,
  P{\'o}cz{\'o}s, and Smola]{reddi2016aide}
Sashank~J. Reddi, Jakub Kone{\v{c}}n{\`y}, Peter Richt{\'a}rik, Barnab{\'a}s
  P{\'o}cz{\'o}s, and Alex Smola.
\newblock {AIDE}: Fast and communication efficient distributed optimization.
\newblock \emph{arXiv preprint arXiv:1608.06879}, 2016.

\bibitem[Rockafellar(1970)]{Rockafellar70book}
R.~T. Rockafellar.
\newblock \emph{Convex Analysis}.
\newblock Princeton University Press, 1970.

\bibitem[Scaman et~al.(2017)Scaman, Bach, Bubeck, Lee, and
  Massouli{\'e}]{ScamanBach2017}
Kevin Scaman, Francis Bach, S{\'e}bastien Bubeck, Yin~Tat Lee, and Laurent
  Massouli{\'e}.
\newblock Optimal algorithms for smooth and strongly convex distributed
  optimization in networks.
\newblock In \emph{Proceedings of the International Conference on Machine
  Learning (ICML)}, pages 3027--3036, 2017.

\bibitem[Shalev-Shwartz(2016)]{shalev2016sdca}
Shai Shalev-Shwartz.
\newblock Sdca without duality, regularization, and individual convexity.
\newblock In \emph{International Conference on Machine Learning}, pages
  747--754, 2016.

\bibitem[Shamir et~al.(2014)Shamir, Srebro, and Zhang]{shamir2014communication}
Ohad Shamir, Nati Srebro, and Tong Zhang.
\newblock Communication-efficient distributed optimization using an approximate
  {N}ewton-type method.
\newblock In \emph{International Conference on Machine Learning}, pages
  1000--1008, 2014.

\bibitem[Tropp(2015)]{tropp2015introduction}
Joel~A. Tropp.
\newblock An introduction to matrix concentration inequalities.
\newblock \emph{Foundations and Trends in Machine Learning}, 8\penalty0
  (1-2):\penalty0 1--230, 2015.

\bibitem[Vershynin(2019)]{vershynin2019high}
Roman Vershynin.
\newblock \emph{High-Dimensional Probability, An Introduction with Applications
  in Data Science}.
\newblock Cambridge University Press, 2019.

\bibitem[Wang and Zhang(2019)]{wang2019utilizing}
Jialei Wang and Tong Zhang.
\newblock Utilizing second order information in minibatch stochastic variance
  reduced proximal iterations.
\newblock \emph{Journal of Machine Learning Research}, 20\penalty0
  (42):\penalty0 1--56, 2019.

\bibitem[Wang et~al.(2018)Wang, Roosta-Khorasani, Xu, and
  Mahoney]{wang2018giant}
Shusen Wang, Farbod Roosta-Khorasani, Peng Xu, and Michael~W Mahoney.
\newblock {GIANT}: Globally improved approximate {N}ewton method for
  distributed optimization.
\newblock In \emph{Advances in Neural Information Processing Systems}, pages
  2332--2342, 2018.

\bibitem[Xiao et~al.(2019)Xiao, Yu, Lin, and Chen]{dscovr2019}
Lin Xiao, Adams~Wei Yu, Qihang Lin, and Weizhu Chen.
\newblock {DSCOVR}: Randomized primal-dual block coordinate algorithms for
  asynchronous distributed optimization.
\newblock \emph{Journal of Machine Learning Research}, 20\penalty0
  (43):\penalty0 1--58, 2019.

\bibitem[Yu et~al.(2010)Yu, Lo, Hsieh, Lou, McKenzie, Chou, Chung, Ho, Chang,
  Wei, et~al.]{yu2010feature}
Hsiang-Fu Yu, Hung-Yi Lo, Hsun-Ping Hsieh, Jing-Kai Lou, Todd~G McKenzie,
  Jung-Wei Chou, Po-Han Chung, Chia-Hua Ho, Chun-Fu Chang, Yin-Hsuan Wei,
  et~al.
\newblock Feature engineering and classifier ensemble for {KDD} cup 2010.
\newblock In \emph{KDD Cup}, 2010.

\bibitem[Yuan and Li(2019)]{yuan2019convergence}
Xiao-Tong Yuan and Ping Li.
\newblock On convergence of distributed approximate {N}ewton methods:
  Globalization, sharper bounds and beyond.
\newblock \emph{arXiv preprint arXiv:1908.02246}, 2019.

\bibitem[Zhang and Xiao(2015)]{zhang2015disco}
Yuchen Zhang and Lin Xiao.
\newblock {DiSCO}: Distributed optimization for self-concordant empirical loss.
\newblock In \emph{International Conference on Machine Learning}, pages
  362--370, 2015.

\bibitem[Zhang and Xiao(2018)]{ZhangXiao2018DiSCO}
Yuchen Zhang and Lin Xiao.
\newblock Communication-efficient distributed optimization of self-concordant
  empirical loss.
\newblock In \emph{Large-Scale and Distributed Optimization}, number 2227 in
  Lecture Notes in Mathematics, chapter~11, pages 289--341. Springer, 2018.

\end{thebibliography}

\newpage 

\appendix

\centerline{\raisebox{-4ex}{\huge Appendix}}
\vspace{2ex}

\section{Convergence Analysis of SPAG}
\label{app:amd}
This section provides proofs for Lemma~\ref{lemma:Gt-bound}, Theorem~\ref{thm:spag} and Lemma~\ref{lemma:LS} presented in Section~\ref{sec:spag}.
Before getting to the proofs, we first comment on the nature of the accelerated convergence rate obtained in Theorem~\ref{thm:spag}.

Note that SPAG (Algorithm~\ref{algo:spag}) can be considered as an accelerated variant of the general mirror descent method considered by \citet{bauschke2017descent} and \citet{lu2018relatively}. 
Specifically, we can replace $D_\phi$ by the Bregman divergence of any convex function of Legendre type \citep[][Section~26]{Rockafellar70book}. 
Recently, \citet{dragomir2019optimal} show that fully accelerated convergence rates, as those for Euclidean mirror-maps achieved by \citet{Nesterov04book}, may not be attainable in the general setting. However, this negative result does not prevent us from obtaining better accelerated rates in the preconditioned setting. Indeed, we choose a smooth and strongly convex mirror map and further assume Lipschitz continuity of its Hessian. For smooth and strongly convex cost functions, the convergence rates of SPAG are almost always better than those obtained by standard accelerated algorithms (without preconditioning) as long as $n$ is not too small, and can be much better with a good preconditioner.

\subsection{Proof of Lemma~\ref{lemma:Gt-bound}}
Using the second-order Taylor expansion (mean-value theorem), we have
\[
    D_\phi(x,y) = \phi(x) - \phi(y) - \langle\nabla \phi(y), x-y\rangle
    = \frac{1}{2}(x-y)^\top\nabla^2 \phi\bigl(y+t(x-y)\bigr)(x-y),
\]
for some scalar $t\in[0,1]$.
We define
\[
    H(x,y) =\nabla^2 \phi\bigl(y+t(x-y)\bigr),
\]
where the dependence on $t\in[0,1]$ is made implicit with the ordered pair 
$(x,y)$. Then we can write
\[
    D_\phi(x,y) = \frac{1}{2} \|x-y\|^2_{H(x,y)}.
\]
By Assumption~\ref{asmp:rel-smooth-sc}, $\phi$ is $L_\phi$-smooth and $\sigma_\phi$-strongly convex, which implies that for all $x,y\in\R^d$,
\[
\sigma_\phi\|x-y\|^2 \leq \|x-y\|^2_{H(x,y)} \leq L_\phi\|x-y\|^2 .
\]
Let $w_t=(1-\beta_t)v_t + \beta_t y_t$. Then we have $x_{t+1}-y_t=\alpha_t\bigl(v_{t+1}-w_t\bigr)$ and
\begin{align*}
    D_\phi(x_{t+1}, y_t)&= \frac{1}{2}\|x_{t+1} - y_t\|^2_{H(x_{t+1},y_t)} \\
    &\leq \frac{L_\phi}{2}\|x_{t+1} - y_t\|^2 =  \alpha_t^2 \frac{L_\phi}{2}\|v_{t+1} - w_t\|^2.
\end{align*}
Next we use $v_{t+1}-w_t=(1-\beta_t)(v_{t+1}-v_t)+\beta_t(v_{t+1}-y_t)$ and convexity of $\|\cdot\|^2$ to obtain
\begin{align*}
    D_\phi(x_{t+1}, y_t)
    &\leq \alpha_t^2\frac{L_\phi}{2}\left( (1 - \beta_t)\|v_{t+1} - v_t\|^2 + \beta_t\|v_{t+1} - y_t\|^2\right)\\
    &\leq \alpha_t^2\frac{L_\phi}{2\sigma_\phi}\left( (1 - \beta_t)\|v_{t+1} - v_t\|^2_{H(v_{t+1},v_t)} + \beta_t\|v_{t+1} - y_t\|^2_{H(v_{t+1},y_t)}\right)\\
    &= \alpha_t^2 \kappa_\phi \bigl((1 - \beta_t)D_\phi(v_{t+1}, v_t) + \beta_t D_\phi(v_{t+1}, y_t)\bigr).
\end{align*}
This finishes the proof of Lemma~\ref{lemma:Gt-bound}.

\subsection{Proof of Theorem~\ref{thm:spag}}

Theorem~2 is a direct consequence of the following result, which is adapted from  \citet{nesterov2017efficiency}.

\begin{theorem}[Smooth and strongly convex mirror map $\phi$]
\label{thm:amd_convergence}
Suppose Assumption~\ref{asmp:rel-smooth-sc} holds. Then the sequences generated by Algorithm~\ref{algo:spag} satisfy for all $t\geq 0$,
\[
    A_t \bigl(\Phi(x_t) - \Phi(x_*)\bigr) + B_t D(x_*, v_t) 
    \leq A_0\bigl(F(x_0) - F(x_*)\bigr) + B_0 D(x_*, v_0).
\]
Moreover, if we set $A_0 = 0$ and $B_0 = 1$ then for $t\geq 0$,
$$
A_t \geq \frac{1}{4 \sigma_\rel}\left[\pi_t^+ - \pi_t^-\right]^2,\qquad
B_t = 1 + \sigma_\rel A_t \geq \frac{1}{4}\left[\pi_t^+ + \pi_t^-\right]^2,
$$
where
$$
\pi_t^+ = \prod_{i=0}^{t-1} \left(1 + \sqrt{\frac{\sigma_\rel}{L_\rel G_t}}\right), \qquad
\pi_t^- = \prod_{i=0}^{t-1} \left(1 - \sqrt{\frac{\sigma_\rel}{L_\rel G_t}}\right).
$$
\end{theorem}

We first state an equivalent definition of relative smoothness and relative strong convexity
\citep{lu2018relatively}.
The function~$F$ is said to be $L_{F/\phi}$-smooth and 
$\sigma_{F/\phi}$-strongly convex with respect to~$\phi$ if for all $x, y\in\R^d$,
\begin{align}
\label{eqn:rel-ssc-approx}
F(y)+\nabla F(y)^\top(x-y) + \sigma_{L/\phi}D_\phi(x,y)
    ~\leq~ F(x) ~\leq~ 
F(y)+\nabla F(y)^\top(x-y) + L_{L/\phi}D_\phi(x,y) .
\end{align}
Obviously this is the same as~\eqref{eqn:rel-ssc-div}.
We also need the following lemma, which is an extension of a result from \citet[Lemma~3.2]{ChenTeboulle93},
whose proof we omit.
\begin{lemma}[Descent property of Bregman proximal point]
    \label{lemma:three-point-descent}
    Suppose $g$ is a convex function defined over $\dom\phi$ and 
    \[
        v_{t+1} = \argmin_{x} \bigl\{ g(x) + (1-\beta_t)D_\phi(x, v_t)+\beta_t D_\phi(x,y_t) \bigr\},
    \]
    then for any $x\in\dom h$, 
    \[
    g(v_{t+1}) + (1-\beta_t)D_\phi(v_{t+1}, v_t)+\beta_t D_\phi(v_{t+1},y_t)
    \leq g(x) + (1-\beta_t)D_\phi(x, v_t)+\beta_t D_\phi(x,y_t) - D_\phi(x,v_{t+1}).
    \]
\end{lemma}

\begin{proof}[Proof of Theorem~\ref{thm:amd_convergence}.]
The proof follows the same lines as~\citet{nesterov2017efficiency}, with adaptations to use general Bregman divergences. 
Applying Lemma~\ref{lemma:three-point-descent} with 
$g(x)=\eta_t\bigl(\nabla f(y_t)^\top x+\psi(x)\bigr)$,
we have for any $x\in\dom\phi$,
\begin{align*}
& D(x, v_{t+1}) + (1 - \beta_t)D(v_{t+1} , v_t) + \beta_t D(v_{t+1}, y_t) - (1-\beta_t) D(x, v_t) - \beta_t D(x, y_t) \\
\leq~&  \eta_t \nabla f(y_t)^\top (x - v_{t+1}) + \eta_t\bigl(\psi(x) - \psi(v_{t+1})\bigr)
.
\end{align*}
Since by definition $\eta_t = \frac{a_{t+1}}{B_{t+1}}$, 
multiplying both sides of the above inequality by $B_{t+1}$ yields
\begin{align*}
B_{t+1} D(x, v_{t+1}) &+ B_{t+1}\bigl((1 - \beta_t)D(v_{t+1} , v_t) + \beta_t D(v_{t+1}, y_t)\bigr) - B_{t+1}(1-\beta_t) D(x, v_t) \\ 
&- B_{t+1}\beta_t D(x, y_t) \leq  a_{t+1} \nabla f(y_t)^\top (x - v_{t+1}) + a_{t+1}\bigl(\psi(x) - \psi(v_{t+1})\bigr).
\end{align*}
Using the scaling property~\eqref{eqn:Bregman-scaling} and the relationships
$\alpha_t=\frac{a_{t+1}}{A_{t+1}}$ and
$a_{t+1}^2 L_{f/\phi} G_t = A_{t+1} B_{t+1}$, we obtain
\begin{align*}
B_{t+1}\bigl((1 - \beta_t)D(v_{t+1} , v_t) + \beta_t D(v_{t+1}, y_t)\bigr) 
&\geq\frac{B_{t+1}}{\alpha_t^2 G_t} D(x_{t+1},y_t) \\
&=\frac{A_{t+1}^2 B_{t+1}}{a_{t+1}^2 G_t} D(x_{t+1},y_t)
=A_{t+1} L_{f/\phi} D(x_{t+1},y_t) .
\end{align*}
Combining the last two inequalities and using the facts 
$B_{t+1}(1-\beta_t)=B_t$ and $B_{t+1}\beta_t=a_{t+1}\sigma_{f/\phi}$, 
we arrive at
\begin{align}
    & B_{t+1} D(x, v_{t+1}) + A_{t+1}L_{f/\phi}D(x_{t+1}, y_t) - B_t D(x, v_t) - a_{t+1}\sigma_{f/\phi} D(x, y_t) \nonumber \\
    \leq~&  a_{t+1} \nabla f(y_t)^\top (x - v_{t+1}) + a_{t+1}\bigl(\psi(x) - \psi(v_{t+1})\bigr) .
    \label{eqn:reduced-expansion}
\end{align}

We then expand the gradient term on the right-hand side of~\eqref{eqn:reduced-expansion} into two parts: 
\begin{align}\label{eqn:inner-prod-expansion}
    a_{t+1} \nabla f(y_t)^\top (x - v_{t+1}) 
    = a_{t+1} \nabla f(y_t)^\top (x - w_t)
    + a_{t+1} \nabla f(y_t)^\top (w_t - v_{t+1}) ,
\end{align}
where $w_t=(1-\beta_t)v_t + \beta_t y_t$.
For the first part,
\begin{align}
    a_{t+1}\nabla f(y_t)^\top(x - w_t) &= a_{t+1}\nabla f(y_t)^\top(x - y_t) + \frac{a_{t+1}(1 - \alpha_t)}{\alpha_t}\nabla f(y_t)^\top(x_t - y_t) \nonumber \\
    &\leq a_{t+1}\left(f(x) - f(y_t) - \sigma_{f / \phi} D(x, y_t)\right) + \frac{a_{t+1}(1 - \alpha_t)}{\alpha_t}\left(f(x_t) - f(y_t)\right).
    \label{eqn:inner-prod-1}
\end{align}
Notice that 
\[
    a_{t+1} \frac{1-\alpha_t}{\alpha_t} 
    = a_{t+1}\left(\frac{1}{\alpha_t}-1\right)
        = a_{t+1}\left(\frac{A_{t+1}}{a_{t+1}}-1\right)
        = A_{t+1}-a_{t+1} = A_t .
\]
Therefore, Equation~\eqref{eqn:inner-prod-1} becomes
\begin{equation}
\label{eq:strong_convexity}
    a_{t+1} \nabla f(y_t)^\top(x - w_t) \leq a_{t+1} f(x) - A_{t+1}f(y_t) + A_t f(x_t) - a_{t+1} \sigma_{f/\phi} D(x, y_t).
\end{equation}
For the second part on the right-hand side of~\eqref{eqn:inner-prod-expansion},
\begin{align}
    & a_{t+1}\nabla f(y_t)^\top(w_t -v_{t+1}) 
    = - \frac{a_{t+1}}{\alpha_t}\nabla f(y_t)^\top(x_{t+1} - y_t) 
    = - A_{t+1} \nabla f(y_t)^\top(x_{t+1} - y_t)  \nonumber \\
    \leq~&  - A_{t+1} \left( f(x_{t+1})-f(y_t) - L_{f/\phi} D(x_{t+1}, y_t)\right),
    \label{eqn:inner-prod-2}
\end{align}
where in the last inequality we used the relative smoothness assumption in~\eqref{eqn:rel-ssc-approx}. 

Summing the inequalities~\eqref{eqn:reduced-expansion}, \eqref{eqn:inner-prod-1}
and~\eqref{eqn:inner-prod-2}, we have
\begin{align*}
    B_{t+1}D(x, v_{t+1}) - B_t D(x, v_t) 
     & \leq  a_{t+1} f(x) - A_{t+1} f(x_{t+1}) + A_t f(x_t) 
    + a_{t+1} (\psi(v_{t+1}) - \psi(x)) \\
    & \leq - A_{t+1} \bigl(f(x_{t+1})-f(x)\bigr) + A_t \bigl(f(x_t)-f(x)\bigr) 
    + a_{t+1} \bigl(\psi(x) - \psi(v_{t+1})\bigr),
\end{align*}
which is the same as
\begin{align}
\label{eqn:f-D-psi}
  A_{t+1} \bigl(f(x_{t+1})-f(x)\bigr) +  B_{t+1}D(x, v_{t+1}) 
  \leq A_t \bigl(f(x_t)-f(x)\bigr) +  B_t D(x, v_t) 
    + a_{t+1} \bigl(\psi(x) - \psi(v_{t+1})\bigr).
\end{align}

Finally we consider the term
$a_{t+1} \bigl(\psi(x) - \psi(v_{t+1})\bigr)$.
Using $x_{t+1} = (1-\alpha_t) x_t + \alpha_t v_{t+1}$ and convexity of $\psi$,
we have
\[
    \psi(x_{t+1}) \leq (1-\alpha_t)\psi(x_t) + \alpha_t \psi(v_{t+1}).
\]
Since by definition $\alpha_t=\frac{a_{t+1}}{A_{t+1}}$ 
and $(1-\alpha_t)=\frac{A_t}{A_{t+1}}$, the above inequality is equivalent to
\[
    A_{t+1} \psi(x_{t+1}) \leq A_t \psi(x_t) + a_{t+1} \psi(v_{t+1}),
\]
which implies (using $A_{t+1} = A_t+a_{t+1}$) that for any $x\in\dom\phi$,
\begin{align}\label{eqn:psi-combination}
    A_{t+1} \bigl(\psi(x_{t+1})-\psi(x)\bigr) 
    \leq A_t \bigl(\psi(x_t)-\psi(x)\bigr) + a_{t+1} \bigl(\psi(v_{t+1})-\psi(x)\bigr).
\end{align}
Summing the inequalities~\eqref{eqn:f-D-psi} and~\eqref{eqn:psi-combination} and
using $\Phi=f+\psi$, we have
\begin{align*}
  A_{t+1} \bigl(\Phi(x_{t+1})-\Phi(x)\bigr) +  B_{t+1}D(x, v_{t+1}) 
  \leq A_t \bigl(\Phi(x_t)-\Phi(x)\bigr) +  B_t D(x, v_t) .
\end{align*}
This can then be unrolled, 
and we obtain the desired result by setting $x=x_*$.

\smallskip

Finally, the estimates of $A_t$ and $B_t$ follow from a direct adaptation of the techniques in~\citep{nesterov2017efficiency}. The only difference is the use of time-varying $\gamma_t=\sqrt{\sigma_{F/\phi}/(L_{F/\phi}G_t)}$ instead of a constant~$\gamma=\sqrt{\sigma_{F/\phi}/L_{F/\phi}}$, which does not impact the derivations. 
\end{proof}

\subsection{Proof of Lemma~\ref{lemma:LS}}

The analysis in Lemma~\ref{lemma:Gt-bound} is very pessimistic, since we use uniform lower and upper bounds for the Hessian of~$\phi$, whereas what we actually want is to bound is the differences between Hessians. If the Hessian is well-behaved (typically Lipschitz, or if $\phi$ is self-concordant), we can prove Lemma~\ref{lemma:LS}, which leads to a finer asymptotic convergence rate.

We start with the local quadratic representation of Bregman divergence:
\begin{align*}
    D_\phi(x_{t+1}, y_t) 
     &= \frac{1}{2}\|x_{t+1} - y_t\|^2_{H(x_{t+1},y_t)} = \frac{\alpha_t^2}{2}\|v_{t+1} - w_t\|^2_{H(x_{t+1},y_t)}\\
     &\leq \frac{\alpha_t^2}{2}\left( (1 - \beta_t)\|v_{t+1} - v_t\|^2_{H(x_{t+1},y_t)} + \beta_t\|v_{t+1} - y_t\|^2_{H(x_{t+1},y_t)} \right)\\
     &\leq \frac{\alpha_t^2}{2}\left((1 - \beta_t)\|v_{t+1} - v_t\|^2_{H(v_{t+1},v_t)} + \beta_t\|v_{t+1} - y_t\|^2_{H(v_{t+1},y_t)} \right)\\
     &\quad + \frac{\alpha_t^2}{2} (1 - \beta_t)\|H(x_{t+1},y_t)-H(v_{t+1},v_t)\|\cdot \|v_{t+1} - v_t\|^2 \\
     &\quad + \frac{\alpha_t^2}{2}\beta_t\|H(x_{t+1},y_t)-H(v_{t+1},y_t)\| \cdot\|v_{t+1} - y_t\|^2 .
\end{align*}
Now we use the Lipschitz property of $\nabla^2\phi$ to bound the spectral norms of differences of Hessians:
\[
\|H(x_{t+1},y_t)-H(v_{t+1},v_t)\| \leq M \|z_{xy}-z_{vv}\|, \qquad
\|H(x_{t+1},y_t)-H(v_{t+1},y_t)\| \leq M \|z_{xy}-z_{vy}\|,
\]
where $z_{vv} \in [v_{t+1}, v_t]$, $z_{xy} \in [y_t, x_{t+1}]$ and $z_{vy} \in [y_t, v_{t+1}]$.
Using the triangle inequality of norms, we have
$$\|z_{xy} - z_{vy}\| = \|z_{xy} - y_t + y_t - z_{vy}\| \leq \|z_{xy} - y_t\| + \|y_t - z_{vy}\| \leq \|x_{t+1} - y_t\| + \|y_t - v_{t+1}\|,$$
and
$$\|z_{vv} - z_{xy}\| \leq \|z_{vv} - v_{t+1}\| + \|v_{t+1} - y_t\| + \|y_t - z_{xy}\| \leq  \|v_t - v_{t+1}\| + \|v_{t+1} - y_t\| + \|y_t - x_{t+1}\|.$$
Therefore, we have
$$d_t \triangleq  \max \bigl\{\|z_{xy} - z_{vv}\|, ~\|z_{xy} - z_{vy}\| \bigr\} \leq \|v_t - v_{t+1}\| + \|v_{t+1} - y_t\| + \|y_t - x_{t+1}\|,$$
and consequently,
\begin{align*}
    D_\phi(x_{t+1}, y_t) 
     &\leq \frac{\alpha_t^2}{2}\left((1 - \beta_t)\|v_{t+1} - v_t\|^2_{H(v_{t+1},v_t)} + \beta_t\|v_{t+1} - y_t\|^2_{H(v_{t+1},y_t)} \right)\\
     & \quad + \frac{M d_t\alpha_t^2}{2}\Bigl( (1 - \beta_t)\|v_{t+1} - v_t\|^2 + \beta_t\|v_{t+1} - y_t\|^2 \Bigr) \\
     &\leq \frac{\alpha_t^2}{2}\left((1 - \beta_t)\|v_{t+1} - v_t\|^2_{H(v_{t+1},v_t)} + \beta_t\|v_{t+1} - y_t\|^2_{H(v_{t+1},y_t)} \right)\\
     & \quad + \frac{M d_t\alpha_t^2}{2\sigma_\phi}\Bigl( (1 - \beta_t)\|v_{t+1} - v_t\|^2_{H(v_{t+1},v_t)} + \beta_t\|v_{t+1} - y_t\|^2_{H(v_{t+1},y_t)}  \Bigr) \\
     &= \frac{\alpha_t^2}{2} \left(1 + \frac{M d_t}{\sigma_\phi}\right) \left((1 - \beta_t)\|v_{t+1} - v_t\|^2_{H(v_{t+1},v_t)} + \beta_t\|v_{t+1} - y_t\|^2_{H(v_{t+1},y_t)} \right)\\
     &= \alpha_t^2 \left(1 + \frac{M d_t}{\sigma_\phi}\right) \Bigl( (1 - \beta_t)D(v_{t+1}, v_t) + \beta_t D(v_{t+1}, y_t) \Bigr).
\end{align*}
Combining with Lemma~\ref{lemma:Gt-bound}, we see that $G_t=\min\{\kappa_\sigma,~1+(M/\sigma_\phi)d_t\}$ satisfies the inequality~\eqref{eqn:Bregman-scaling}.
This finishes the proof of Lemma~\ref{lemma:LS}.

\smallskip

Note that this condition is not directly useful. Indeed, $x_{t+1}$ and $v_{t+1}$ depend on $G_t$.  Yet, under the uniform choice of $G_t\leq\kappa_\phi$, it can be shown that $d_t \rightarrow 0$ at rate $(1-1/\sqrt{\kappa_\phi \kappa_\rel})^t$ because the sequences $v_t$, $x_t$ and $y_t$ all converge to $x^*$ at this rate in the strongly convex case \citep[][Theorem~1]{LinXiao2015homotopy}. 
As a consequence, Algorithm~\ref{algo:spag} will eventually use $G_t \leq 2$, leading to an asymptotic rate of $(1-1/\sqrt{\kappa_\rel})^t$.

\section{Concentration of Hessians}
\label{app:concentration}

In practice, preconditioned gradient methods such as DANE are often used with a step-size of $1$. This implies the assumption of $L_\rel = 1$, which holds if~$n$ is sufficiently large with a given~$\mu$ or if~$\mu$ is sufficiently large for a given~$n$ (but $\mu\leq L_F$ always). Otherwise  convergence is not guaranteed (which is why it is sometimes considered as ``rather unstable"). If $\mu$ is such that $\|H_f(x) - H_F(x)\| \leq \mu$ for all $x\in \cB(0, D)$ then $L_\rel = 1$ can safely be chosen since $H_F(x) - H_f(x) \preceq \mu I_d$. Note that this choice of $\mu$ is completely independent of $\lambda$. In this case, we use that $H_F(x) - H_f(x) \succeq - \mu I$ to write that
$$H_f(x) + \mu \preceq H_F(x) + 2\mu \preceq (1 + 2\mu H_F^{-1}(x)) H_F(x) \preceq \left(1 + \frac{2\mu}{\lambda}\right) H_F(x).$$
These derivations are similar to the ones of \citet[Lemma~3]{ZhangXiao2018DiSCO}, and so we obtain $\sigma_\rel = \left(1 + \frac{2\mu}{\lambda}\right)^{-1}$ and the corresponding relative condition number $\kappa_\rel = 1 + \frac{2\mu}{\lambda}$, as explained in Section~\ref{sec:stat-precond}. We see that $\mu$ is independent of $\lambda$, but the problem is still very ill-conditioned for small values of $\lambda$, meaning that acceleration makes a lot of sense. In the quadratic case, tighter relative bounds can be derived. 

\subsection{The quadratic case}
\label{app:concentration_quad}
This section is focusing on proving Theorem~\ref{thm:concentration_quadratics}.

\begin{proof}[Proof of Theorem~\ref{thm:concentration_quadratics}]
We consider the random variable $a$, and $(a_i)_{i \in \{1, ..., n\}}$ are $n$ i.i.d. variables with the same law as $a$.  We introduce matrices $\hat{H}$ and $H$ such that $H_f = \hat{H} + \lambda I_d$ and $H_F = H + \lambda I_d$. In particular, $H =  \esp{aa^\top} = \dE \hat{H}$.
We define for $\alpha \geq 0$, $\beta > 0$, $H_{\alpha, \beta} = \alpha H + \beta I_d$, and $$S_i = \frac{1}{n}H_{\alpha, \beta}^{-\frac{1}{2}}(a_i a_i^\top - H) H_{\alpha, \beta}^{-\frac{1}{2}},$$
which is such that $\esp{S_i} = 0$. This allows to have bounds of the form $\|\sum_i S_i \| \leq t$ with probability $1 - \delta$ and a spectral bound~$\mu$ that depends on $\alpha$, $\beta$, $\delta$ (and other quantities related to $H$ and $a_i a_i^\top$). We note that 
$$\sum_{i=1}^n S_i = H_{\alpha, \beta}^{-\frac{1}{2}}(\hat{H} - H)H_{\alpha, \beta}^{-\frac{1}{2}},$$
and write the concentration bounds on the $S_i$ as $- t H_{\alpha, \beta} \preceq \hat{H} - H \preceq t H_{\alpha, \beta}$ for some $t>0$, which can be rearranged as:
\begin{align*}
&    \hat{H} + t \beta I_d \succeq (1 - t \alpha) H \\
& \hat{H} - t \beta I_d \preceq (1 + t \alpha) H.
\end{align*}
Using $H_f = \hat{H} + \lambda I_d$ and $H_F = H + \lambda I_d$, the first equation can be rearranged as:
\begin{equation}
H_F \preceq \frac{1}{1 - t \alpha}\left(H_f + t (\beta - \alpha \lambda) I_d\right).
\end{equation}
The second equation can be written
$$H_f \preceq \left[(1 + t\alpha)I_d + t(\beta - \alpha \lambda)H_F^{-1}\right]H_F,$$
which, by adding $t(\beta - \alpha \lambda)I_d$ on both sides, leads to
$$H_f + t(\beta - \alpha \lambda) I_d \preceq \left[(1 + t\alpha)I_d + 2t(\beta - \alpha \lambda)H_F^{-1}\right]H_F.$$
We let $\mu = t(\beta - \alpha \lambda)$ and use $H_F^{-1} \preceq \lambda^{-1} I_d$ to write that:
\begin{equation}
\label{eq:generic_concentration_quadratics}
\left(1 + \alpha t + \frac{2 \mu}{\lambda}\right)^{-1} (H_f + \mu I_d) \preceq H_F \preceq \frac{1}{1 - \alpha t}\left(H_f + \mu I_d \right).
\end{equation}
We then use the fact that $a_i a_i^\top$ and $H$ are positive semidefinite and upper bounded by $R^2 I$ to write that: 
\begin{equation}\label{eqn:Si-bound}
\| S_i \| \leq \frac{1}{n}\|H_{\alpha, \beta}^{-1}\| \max\bigl\{\|a a^\top \|, \| H \|\bigr\} \leq \frac{R^2}{\beta n}.
\end{equation}
Using the fact that $H = \esp{a a^\top}$, we bound the variance as:
\begin{align*}
    \Bigl\|\sum_i\esp{ S_i S_i^\top}\Bigr\| &= \frac{1}{n}\left\|\esp{H_{\alpha, \beta}^{-\frac{1}{2}}(a a^\top - H) H_{\alpha, \beta}^{-1}(aa^\top - H) H_{\alpha, \beta}^{-\frac{1}{2}}}\right\| \\
    &= \frac{1}{n}\left\|H_{\alpha, \beta}^{-\frac{1}{2}}(\esp{a a^\top H_{\alpha, \beta}^{-1} a a^\top} - H H_{\alpha, \beta}^{-1} H) H_{\alpha, \beta}^{-\frac{1}{2}}\right\|\\
    &\leq \frac{1}{n}\max\left\{\tilde{R}^2\left\|H_{\alpha, \beta}^{-\frac{1}{2}}\esp{aa^\top}H_{\alpha, \beta}^{-\frac{1}{2}}\right\|,~\left\|H_{\alpha, \beta}^{-\frac{1}{2}} H H_{\alpha, \beta}^{-1} H H_{\alpha, \beta}^{-\frac{1}{2}}\right\|\right\}\\
    &\leq \frac{1}{n}\left\|H_{\alpha, \beta}^{-\frac{1}{2}}H H_{\alpha, \beta}^{-\frac{1}{2}}\right\| \max\left\{\tilde{R}^2, ~\left\|H_{\alpha, \beta}^{-\frac{1}{2}}H H_{\alpha, \beta}^{-\frac{1}{2}}\right\|\right\},
\end{align*}
with $\tilde{R}^2 \geq a^\top H_{\alpha, \beta}^{-1} a$ almost surely. We first notice that $a_i^\top H_{\alpha, \beta}^{-1} a_i \leq \frac{R^2}{\beta}$. Then, we use the positive definiteness of $H_{\alpha, \beta}$ and $H$ and the fact that $\beta  H_{\alpha, \beta}^{-1} \preceq I_d$ to show that for $\alpha > 0$:
\begin{align*}
    \left\|H_{\alpha, \beta}^{-\frac{1}{2}}H H_{\alpha, \beta}^{-\frac{1}{2}}\right\| &= \left\|H_{\alpha, \beta}^{-\frac{1}{2}}\frac{(\alpha H + \beta - \beta)}{\alpha} H_{\alpha, \beta}^{-\frac{1}{2}}\right\| = \frac{1}{\alpha}\left\|I_d - \beta  H_{\alpha, \beta}^{-1}\right\|
    \leq \alpha^{-1}\left(1 - \frac{\beta}{\alpha L + \beta}\right)
    = \frac{L}{\alpha L + \beta},
\end{align*}
where $L$ is the spectral norm of~$H$, i.e., $L=\|H\|$.
A quick calculation shows that this formula is also true for $\alpha=0$. In the case $\alpha = 0$ and $\beta = 1$ (absolute bounds), $H_{\alpha, \beta} = I_d$ and we recover that we can bound the variance by $\frac{L R^2}{n}$, leading to the usual additive bounds.

For $\alpha > 0$, we use the simpler bound $\bigl\|H_{\alpha, \beta}^{-\frac{1}{2}}HH_{\alpha, \beta}^{-\frac{1}{2}}\bigr\| \leq \alpha^{-1}$ and $\tilde{R}^2\leq \beta^{-1} R^2$, leading to
$$\Bigl\|\sum_i\esp{ S_i S_i^\top}\Bigr\| \leq \frac{\max(\beta^{-1}R^2, \alpha^{-1})}{n \alpha}.$$
For any $1 > \delta > 0$, we note $c_\delta = \frac{28}{3}\log\left(\frac{2d}{\delta}\right)$. We now set $\alpha = \frac{\beta n}{c_\delta R^2}$, and assume that $n > c_\delta$ (otherwise concentration bounds will be very loose anyway). In this case, $\beta^{-1}R^2 \geq \alpha^{-1}$, meaning that the bound on the variance becomes:
$$\Bigl\|\sum_i\esp{ S_i S_i^\top}\Bigr\| \leq \frac{1}{\alpha^2 c_\delta}.$$
Similarly, according to~\eqref{eqn:Si-bound}, every $S_i$ is almost surely bounded as:
$\| S_i \| \leq \frac{1}{\alpha c_\delta}$.
We can now use Matrix Bernstein Inequality ~\citep[][Theorem (6.1.1)]{tropp2015introduction} to get that with probability $1 - p_\delta$ and for $t \geq 0$, $$\Bigl\|\sum_{i=1}^n S_i\Bigr\| \leq t,$$
with $$p_\delta = 2d \cdot \exp(-\frac{t^2 / 2}{(\alpha^2 c_\delta)^{-1} + (\alpha c_\delta)^{-1} t / 3}).$$
We choose $t = (2\alpha)^{-1}$, which leads to $p_\delta = \delta$. By substituting the expressions of $\alpha t = \frac{1}{2}$ and $\beta t = \frac{R^2 c_\delta}{n} \alpha t$ into Equation~\eqref{eq:generic_concentration_quadratics}, we obtain:
\begin{equation*}
\left(\frac{3}{2} + \frac{2 \mu}{\lambda}\right)^{-1} (\hat{H}_\lambda + \mu I_d) \preceq H_\lambda \preceq 2 \left(\hat{H}_\lambda + \mu I_d\right),
\end{equation*}
with 
\begin{equation*}
\mu = t(\beta-\alpha\lambda) = \frac{1}{2}\left(\frac{28 R^2}{3n}\log\left(\frac{2d}{\delta}\right) - \lambda \right).
\end{equation*}

In case $\beta$ is very small so that $\mu < 0$ then it is always possible to choose $\delta^\prime < \delta$ so that $\mu > 0$. This means that the same bound on~$\mu$ holds with probability $1 - \delta^\prime > 1 - \delta$.
\end{proof}

\subsection{Almost surely bounded $a$}
\label{app:concentration_gen}

We first introduce Theorem~\ref{thm:sudakov_gen}, which proves a general concentration result that implies Theorem~\ref{thm:concentration_sudakov} as a special case.

\begin{theorem}
\label{thm:sudakov_gen}
We consider functions $\varphi_1$, $\varphi_2$, which are respectively  $L_1$ and $L_2$ Lipschitz-continuous. We consider two sets $\X$ and $\Y$ which are contained in balls of center 0 and radius $D_1$ and $D_2$. We assume that $|\varphi_1(a_i^\top x) | \leq B_1$ and $|\varphi_2(a_i^\top y) | \leq B_2$ almost surely for all $x \in \X$ and $y \in \Y$. We consider 
$$Y = \sup_{x \in \X, ~y \in \Y} \Big\{ \frac{1}{n} \sum_{i=1}^n \varphi_1(a_i^\top x) \varphi_2( a_i^\top y) - \E \varphi_1(a^\top x) \varphi_2( a^\top y) \Big\}.$$
Then, for all $1 \geq \delta > 0$, with probability greater than $1 - \delta$:
   $$
   Y \leq 
   \sqrt{4\pi}\frac{( \esp{\| a\|^2} )^{\frac{1}{2}}}{\sqrt{n}}  (  B_2 L_1  D_1 
 +  B_1 L_2  D_2 )
   + \frac{2B_1 B_2 }{\sqrt{2n}} \sqrt{\log \frac{1}{\delta}} .
   $$
\end{theorem}

Theorem~\ref{thm:concentration_sudakov} is then a direct corollary of Theorem~\ref{thm:sudakov_gen}, as shown below:
\begin{proof}[Proof of Theorem~\ref{thm:concentration_sudakov}]
The result is obtained by applying Theorem~\ref{thm:sudakov_gen} with $\varphi_1 = \ell^{\prime \prime}$ and $\varphi_2 = \frac{1}{2}(\cdot)^2$. This implies that with probability at least $1 - \delta$,
$$ \sup_{x \in \mathcal{B}(0, D),~ y \in \mathcal{B}(0,1)} y^\top \Big[ \frac{1}{n} \sum_{i=1}^n \ell^{\prime \prime}(a_i^\top x) a_i a_i^\top - \E \ell^{\prime \prime}(a^\top x) a a^\top \Big] y \leq \mu, $$
where the value of $\mu$ can be obtained by letting $B_1 = B_\ell$, $L_1 = M_\ell$, $D_1 = D$, $D_2 = 1$, $B_2 = \sup_{y: \|y\| \leq 1} y^\top a_i a_i^\top y \leq R^2$ and $L_2 = \sup_{y: \|y\| \leq 1} 2 \| y^\top a_i \| = 2 R$.
\end{proof}

\begin{proof}[Proof of Theorem~\ref{thm:sudakov_gen}]
If changing any $a_i$ to some $a_i'$, then the deviation in $Y$ is at most (almost surely): 
$$
\frac{1}{n} \sup_{x \in \X, \ y \in \Y} 
\bigl| \varphi_1(a_i^\top x) \varphi_2( a_i^\top y) \bigr|
+ \sup_{x \in \X, \ y \in \Y} 
\bigl|\varphi_1(a_i'^\top x) \varphi_2( a_i'^\top y) \bigr|  \leq \frac{2}{n} B_1 B_2.$$

Mac-Diarmid's inequality~\citep[see, e.g.,][Theorem~2.9.1]{vershynin2019high} thus implies that with probability greater than $1-\delta$,
\begin{equation}
    \label{eq:mc_diarmid}
    Y  \leq \E Y + \frac{2B_1 B_2 }{\sqrt{2n}} \sqrt{\log \frac{1}{\delta}} \ .
\end{equation}

In order to bound $\E Y$, we first use classical symmetrization property~\citep[see, e.g.,][Section~6.4]{vershynin2019high}
$$
\E Y \leq \sqrt{2\pi}  \cdot \E \sup_{x \in \X, \ y \in \Y}  \frac{1}{n} \sum_{i=1}^n \varepsilon_i \varphi_1(a_i^\top x) \varphi_2( a_i^\top y),
$$
where each $\varepsilon_i$ is an independent standard normal variable.

Denoting $Z_{x,y} = \frac{1}{n} \sum_{i=1}^n \varepsilon_i \varphi_1(a_i^\top x) \varphi_2( a_i^\top y)$, we have, for any $x,y,x',y'$, assuming the $a_i$ are fixed,
\BEAS
\E ( Z_{x,y} - Z_{x',y'})^2
& = & \frac{1}{n^2} \sum_{i=1}^n 
\Big( \varphi_1(a_i^\top x) \varphi_2( a_i^\top y) - \varphi_1(a_i^\top x') \varphi_2( a_i^\top y')
 \Big)^2 \\
 & = & \frac{1}{n^2} \sum_{i=1}^n 
\Big( \varphi_1(a_i^\top x) \bigl[ \varphi_2( a_i^\top y) -  \varphi_2( a_i^\top y')\bigr] 
+\bigl[ \varphi_1(a_i^\top x) -  \varphi_1(a_i^\top x') \bigr] \varphi_2( a_i^\top y')
 \Big)^2 \\
 & \leq & \frac{1}{n^2} \sum_{i=1}^n \Big( 
 2\, \varphi_1(a_i^\top x)^2 \bigl[ \varphi_2( a_i^\top y) -  \varphi_2( a_i^\top y')\bigr]^2
+ 2\, \varphi_2( a_i^\top y')^2 \bigl[ \varphi_1(a_i^\top x) -  \varphi_1(a_i^\top x') \bigr]^2 
\Big) \\
 & \leq & \frac{1}{n^2} \sum_{i=1}^n \Big( 
 2 B_1^2 \bigl[  \varphi_2( a_i^\top y) -  \varphi_2( a_i^\top y')\bigr]^2
+ 2B_2^2 \bigl[ \varphi_1(a_i^\top x) -  \varphi_1(a_i^\top x') \bigr ]^2 
\Big) .
\EEAS
We then have, using Lipschitz-continuity:
\BEAS
\E ( Z_{x,y} - Z_{x',y'})^2
& \leq & 
 \frac{1}{n^2} \sum_{i=1}^n \Big( 
 2 B_1^2 L_2^2 \bigl[  a_i^\top y -    a_i^\top y' \bigr]^2
+ 2B_2^2 L_1^2 \bigl[  a_i^\top x  -   a_i^\top x' \bigr]^2 
\Big)
\\
& = & \E ( \tilde{Z}_{x,y} - \tilde{Z}_{x',y'})^2,
 \EEAS
 for $$\tilde{Z}_{x,y} = \frac{1}{n} \sum_{i=1}^n \Big\{ \sqrt{2}B_2 L_1 \tilde{\varepsilon}_{1i}  a_i^\top x
 + \sqrt{2}B_1 L_2 \tilde{\varepsilon}_{2i} a_i^\top y \Big\},$$
 with all $\tilde{\varepsilon}_{1i}$ and $\tilde{\varepsilon}_{2i}$ independent standard random variables.
 
 Using Sudakov-Fernique inequality~\citep[][Theorem  7.2.11]{vershynin2019high}, we get
 \BEAS
 \E Y  & = & \sqrt{2\pi} \E \sup_{x \in \X, \ y \in \Y}  Z_{x,y} \\
 & \leq  & \sqrt{2\pi} \E \sup_{x \in \X, \ y \in \Y}  \tilde{Z}_{x,y} \\
 &= & \sqrt{4\pi}B_2 L_1  \E \sup_{x \in \X} \frac{1}{n} \sum_{i=1}^n \tilde{\varepsilon}_{1i}  a_i^\top x
 + \sqrt{4\pi}B_1 L_2  \E \sup_{y \in \Y} \frac{1}{n} \sum_{i=1}^n \tilde{\varepsilon}_{2i}  a_i^\top y \\
 & \leq & \sqrt{4\pi} B_2 L_1  D_1 \E \Big\|  \frac{1}{n} \sum_{i=1}^n \tilde{\varepsilon}_{1i}  a_i \Big\|
 + \sqrt{4\pi} B_1 L_2  D_2 \E \Big\|  \frac{1}{n} \sum_{i=1}^n \tilde{\varepsilon}_{2i}  a_i \Big\|
\\
 & \leq & \sqrt{4\pi} B_2 L_1  D_1 \sqrt{\E \Big\|  \frac{1}{n} \sum_{i=1}^n \tilde{\varepsilon}_{1i}  a_i \Big\|^2}
 + \sqrt{4\pi} B_1 L_2  D_2 \sqrt{\E \Big\|  \frac{1}{n} \sum_{i=1}^n \tilde{\varepsilon}_{2i}  a_i \Big\|^2}
\\
 & \leq & \sqrt{4\pi} B_2 L_1  D_1 \frac{( \esp{\| a\|^2} )^{\frac{1}{2}}}{\sqrt{n}} 
 + \sqrt{4\pi} B_1 L_2  D_2 \frac{( \esp{\| a\|^2} )^{\frac{1}{2}}}{\sqrt{n}}.
   \EEAS
   
   Plugging this into Equation~\eqref{eq:mc_diarmid}, we obtain that with probability greater than $1-\delta$,
   $$
   Y \leq 
   \sqrt{4\pi}\frac{( \esp{\| a\|^2} )^{\frac{1}{2}}}{\sqrt{n}}  (  B_2 L_1  D_1 
 +  B_1 L_2  D_2 )
   + \frac{2B_1 B_2 }{\sqrt{2n}} \sqrt{\log \frac{1}{\delta}} .
   $$
\end{proof}

\begin{remark}[Relative bounds]
\emph{In the quadratic case, considering relative bounds allowed to choose smaller values of $\mu$ and to tighten the bounds on the relative condition number by a $\sqrt{n}$ factor. Theorem~\ref{thm:concentration_sudakov} consists in bounding (using the definition of the operator norm)
$$\sup_{x\in \cB(0,D), y \in \cB(0,1)} \Big\{
\frac{1}{n}\sum_{i=1}^n\ell^{\prime \prime}(a_i^\top x) (a_i^\top y)^2 -
y^\top H(x) y\Big\},$$
and heavily relies on the fact that $(a_i^\top y)^2$ is independent of $x$. The proof needs to be adapted in the case of the relative bounds since this term becomes $(a_i^\top H_{\alpha, \beta}^{-\frac{1}{2}}(x) y)^2$, which now depends on $x$ as well, and thus requires a different control.}
\end{remark}

\subsection{Subgaussian $a$}
\label{app:concentration_subgaussian}
We considered in the previous section a splitting of the summands of the Hessians as a product of $2$ functions. We now present a different bound that is designed for a product of an arbitrary number of functions $\varphi_1, \ldots, \varphi_r:\dR\to\dR$. This section is devoted to proving Theorem~\ref{thm:subgaussian_gen}, which is based on the chaining argument~\citep[][Chapter~13]{boucheron2013concentration}, and from which Theorem~\ref{thm:concentration_subgaussian} can be derived directly.

\begin{theorem}
\label{thm:subgaussian_gen}
Assume that for all $i$, $\varphi_i(0)=0$ and $\varphi_i$ is $1$-Lipschitz. Assume that $a$ is $\rho$-subgaussian, and that for all $k$, $\sup_{x\in \cB(0,1)}|\varphi_k(a_i^\top x)|\le B_k$. Denote $B=\prod_{k=1}^r B_k$. For suitable constant $C_r$, for all $\gamma>0$, one has that
\begin{align*}
&\dP\left(\sup_{x_1,\ldots,x_r\in \cB(0, 1)} \frac{1}{n}\sum_{i\in[n]}\left\{\prod_{k=1}^r \varphi_k(a_i^\top x_k)-\dE\prod_{k=1}^r\varphi_k(a^\top x_k)\right\}\ge \rho^rC_r(d+\gamma)\left[\frac{1}{\sqrt{dn}}+\frac{(\rho^{-r}B)^{1-2/r}}{n}\right]\right) \\
&\le r\frac{\pi^2}{6} e^{-\gamma}.
\end{align*}
\end{theorem}

We are primarily interested in the case $r=3$, $\varphi_1=\varphi_2=\mathrm{id}$ (the identity mapping) to control distances between Hessians. 

\begin{proof}
We look for bounds on 
\begin{equation}\label{eq:new_bound_2}
Y:=\sup_{x_1,\ldots,x_r\in \cS_1} \frac{1}{n}\sum_{i\in[n]}\left\{\prod_{k=1}^r \varphi_k(a_i^\top x_k)-\dE_a\prod_{k=1}^r\varphi_k(a^\top x_k)\right\}.
\end{equation}

For all $j\ge 0$, let $\cN_j$ be an $\epsilon$-net of $\cS_1$ that approximates $\cS_1$ to distance $2^{-j}$. Then, $\cN_j$ can be chosen as $|\cN_j|\le (1+2^{j+1})^d$\citep[see, e.g.,][Section~4.2]{vershynin2019high}. For all $x\in\cS_1$, let $\Pi_j(x)$ be some point in $\cN_j$ such that $\|x-\Pi_j(x)\|\le 2^{-j}$. By convention we take $\Pi_0(x)=0$.

Then for all $(x_1,\ldots,x_r)\in\cS^r$, using the chaining approach~\citep{boucheron2013concentration}, we write
\begin{align*}
&\frac{1}{n}\sum_{i\in[n]}\prod_{k\in[r]}\varphi_k(a_i^\top x_k) \\
=~&\sum_{j\ge 0}\frac{1}{n}\sum_{i\in[n]}\left\{
\prod_{k\in[r]}\varphi_k(a_i^\top \Pi_{j+1} (x_k))-\prod_{k\in[r]}\varphi_k(a_i^\top \Pi_{j}(x_k))\right\}\\
=~&\sum_{j\ge 0}\sum_{k\in[r]}\frac{1}{n}\sum_{i\in[n]} \prod_{\ell=1}^{k-1}\varphi_{\ell}(a_i^\top\Pi_{j+1}(x_{\ell}))\bigl[\varphi_k(a_i^\top\Pi_{j+1}(x_k))-\varphi_k(a_i^\top\Pi_j(x_k))\bigr]\prod_{\ell=k+1}^r \varphi_\ell(a_i^\top \Pi_j(x_\ell)).
\end{align*}

Let $j\ge 0$ and $k\in[r]$ be fixed.
Consider a term of the form $Z=\frac{1}{n}\sum_{i\in[n]}Z_i$, with 
\begin{equation}\label{eq:z_i}
Z_i=\prod_{\ell=1}^{k-1} \varphi_\ell(a_i^\top u_{\ell})\bigl[\varphi_k(a_i^\top u_k)-\varphi_k(a_i^\top v_k)\bigr]\prod_{\ell=k+1}^r \varphi_\ell(a_i^\top v_\ell),
\end{equation}
where $u_\ell\in\cN_j$, $v_\ell\in\cN_{j+1}$, and $\|u_k-v_k\|\le \epsilon_j:=2^{-j+1}$. By the triangle inequality, for all $x_\ell\in \cS_1$, letting $u_\ell=\Pi_j(x_\ell)$ and $v_\ell=\Pi_{j+1}(x_\ell)$, these assumptions are satisfied. For each $Z_i$ and $t>0$, we have:
\begin{align*}
\dP(Z_i\ge\epsilon_j\rho^rt)
&\le \dP\Bigl(|\varphi_\ell(a_i^\top u_\ell)|\ge \rho t^{1/r}\hbox{ for some }\ell<r,\\
&\qquad~ \hbox{ or }|\varphi_k(a_i^\top u_k)-\varphi_k(a_i^\top v_k)|\ge \rho \epsilon_j t^{1/r},\\
&\qquad~ \hbox{ or }|\varphi_\ell(a_i^\top v_\ell)|\ge \rho t^{1/r}\hbox{ for some }\ell>k \Bigr).  
\end{align*}
Therefore, we have 
$$ \dP(Z_i\ge\epsilon_j\rho^rt) \le 2r e^{-t^{2/r}/2} \hbox{ if } t\le \rho^{-r}P_{j,k} \hbox{ and }
$$
$$
\dP(Z_i\ge\epsilon_j\rho^rt) = 0 \hbox{ if }t> \rho^{-r}P_{j,k},
$$
where we noted $P_{j,k}:=\min\bigl\{2B/\epsilon_j, 2(B/B_k)R\bigr\}$. We will also make use of notation $j^*(k):=\lceil \log_2(R/B_k)\rceil$, so that
$$
j\le j^*(k)\Rightarrow P_{j,k}=2B/\epsilon_j,\qquad j>j^*(k) \Rightarrow P_{j,k}=2(B/B_k)R.
$$

Fixing $j\ge 0$, $k\in[r]$, we write for any $\theta > 0$ (a specific $\theta$ will be chosen later):
$$
\dE\, e^{(\theta/n)\rho^{-r}[Z_i-\dE Z_i]/\epsilon_j}=1+\left(\frac{\theta}{n}\right)^2 \dE \Bigl[(\epsilon_j^{-1}\rho^{-r}(Z_i-\dE Z_i))^2 F((\theta/n)(\epsilon_j^{-1}\rho^{-r}(Z_i-\dE Z_i)))\Bigr], 
$$
where 
$$
F(x):=x^{-2}[e^{x}-x-1]\le e^{|x|}.
$$
Thus using this bound and the inequality $xy\le x^2+y^2$:
\begin{equation}
\label{eq:exp_moment_main}
\dE\, e^{(\theta/n)\rho^{-r}[Z_i-\dE Z_i]/\epsilon_j}\le 1+\left(\frac{\theta}{n}\right)^2\left[\dE((\epsilon_j^{-1}\rho^{-r}(Z_i-\dE Z_i))^4+\dE e^{2(\theta/n)\rho^{-r}|Z_i -\dE Z_i|/\epsilon_j}\right].
\end{equation}

By the sub-gaussian tail assumption, $\dE(\epsilon_j^{-1}\rho^{-r}(Z_i-\dE Z_i))^4$ is bounded by a constant $\kappa_r$ dependent on $r$. We now assume that $\theta$ is such that
$$
\frac{\theta}{n}\le \min\left(\frac{(\rho^{-r}P_{j,k})^{2/r-1}}{8}, 1\right),
$$
which is equivalent to having $(\theta / n)y \le y^{2/r} / 8$ for $y \in [0, \rho^{-r}P_{j,k}]$ and $r \geq 2$. Then, $\dE e^{2(\theta/n)\rho^{-r}|Z_i -\dE Z_i|/\epsilon_j}$ is also bounded by another constant $\kappa'_r$ dependent on $r$. Indeed, by the sub-gaussian tail assumption, $|\dE Z_i|\le \rho^r \epsilon_j  s_r$ for some $r$-dependent constant, and we can then use the fact that:
\begin{align*}
\dE e^{\alpha X} &= \int_{0}^{\infty}e^{kz}p(z)dz \\
&= \int_{0}^{\infty} \left(1 + \alpha \int_{0}^{z} e^{\alpha y}\right) p(z)dz dy \\
&= 1 + \alpha  \int_{0}^{\infty} \int_{y}^{\infty} e^{\alpha y} dy p(z)dz \\
&= 1 + \alpha  \int_{0}^{\infty} e^{\alpha y} p(X \geq y) dy,
\end{align*}
with $\alpha = 2 \theta / n$ and $X = \rho^{-r} |Z_i| \epsilon_j$ to get:
\begin{align*}
\dE\, e^{2(\theta/n)\rho^{-r}|Z_i-\dE Z_i|/\epsilon_j}
&\le \dE\, e^{2(\theta/n)\rho^{-r}(|Z_i| + |\dE Z_i|)/\epsilon_j} \\
&\le e^{2(\theta / n) s_r} \dE e^{2\theta/n\rho^{-r}|Z_i|/\epsilon_j} 
\\
&\le e^{2(\theta/n)s_r}[1+\frac{2\theta}{n}\int_0^{\infty}e^{2(\theta/n)y}[\dP(Z_i\ge y\rho^r \epsilon_j)+\dP(-Z_i\ge y\rho^r\epsilon_j)]dy]
\\
&\le e^{2(\theta/n)s_r}[1+\frac{2\theta}{n}2r\int_0^{\rho^{-r}P_{j,k}}e^{2(\theta/n)y-y^{2/r}/2}dy],\\
&\le e^{2(\theta/n)s_r}[1+\frac{2\theta}{n}2r\int_0^\infty e^{-y^{2/r}/4}dy]\\
&= e^{2(\theta/n)s_r}[1 + \frac{\theta}{n}c_r].
\end{align*}
We finally use the fact that $\theta / n \le 1$ to write $\dE e^{2(\theta/n)\rho^{-r}|Z_i-\dE Z_i|/\epsilon_j} \leq \kappa^\prime_r$, with $\kappa^\prime_r = e^{2s_r}[1 + c_r]$. We write $\kappa^{\prime \prime}_r = \kappa_r + \kappa^\prime_r$ and use Equation~\eqref{eq:exp_moment_main} together with the independence of the $Z_i$ to obtain:
$$
\dE\, e^{\theta \rho^{-r} \left[\frac{1}{n}\sum_{i=1}^n Z_i - \dE Z_i\right]/\epsilon_j} \le \left( 1 + \left(\frac{\theta}{n}\right)^2 \kappa^{\prime \prime}_r\right)^n \le e^{\frac{\theta^2}{n} \kappa^{\prime \prime}_r}.
$$

Thus, using that $\dP(X \geq y) = \dP(e^X \geq e^y) \leq e^{-y} \dE e^X$ (Markov Inequality), we have that for fixed $u_\ell$ $u_\ell,v_{\ell},\; \ell\in[r]$ in the suitable $\epsilon$-nets is upper bounded for all $\theta\in[0,\min(n, n(\rho^{-r} P_{j,k})^{2/r-1}/8)]$ as:
\begin{equation} \label{eq:markov_zi}
\dP\left(\frac{1}{n}\sum_{i\in[n]}Z_i-\dE Z_i\ge \rho^r \epsilon_j t_{j,k}\right) \le \exp\left( (r+1)d\ln(1+ 2^{j+2}) - \theta t_{j,k} + \kappa''_r \theta^2/n\right).
\end{equation}

We see in Equation~\eqref{eq:z_i} that the variables $Z_i$ are built by fixing a specific either $u_\ell$ for $\ell < k$, $v_\ell$ for $\ell > k$, and $u_k$ and $v_k$, meaning that there are actually $r+1$ variables to be fixed in nets of resolution either $2^{-j}$ or $2^{-j - 1}$. Note that all $Z_i$ for $i \in \{1, \cdots, n\}$ are constructed with the same choice of $u_\ell$ and $v_\ell$. Therefore, the number of possible choices for $u_\ell\in\cN_j$ and $v_\ell\in\cN_{j+1}$ involved in the definition of $Z_i$ is upper-bounded by
$$
|\cN_{j+1}|^{r+1}\le e^{d (r+1) \ln(1+2^{j+2})}.
$$
Combining this with Equation~\eqref{eq:markov_zi}, we obtain using a union bound that:
\begin{align*}
\dP\left(\sup_{u_\ell, v_\ell} \Big\{Z-\dE Z \Big\} \ge \rho^r \epsilon_j t_{j,k}\right) &= \dP\left(\cup_{u_\ell, v_\ell} \Big\{Z-\dE Z \ge \rho^r \epsilon_j t_{j,k}\Big\}\right)\\
&\le \sum_{u_\ell, v_\ell} \dP\left(Z-\dE Z \ge \rho^r \epsilon_j t_{j,k}\right)\\
&\le \exp\left( (r+1)d\ln(1+ 2^{j+2}) - \theta t_{j,k} + \kappa''_r \theta^2/n\right).
\end{align*}
Let now $\theta_{j,k}=\min(n, n(\rho^{-r}P_{j,k})^{2/r-1}/8,\sqrt{nd})$, and 
$$
t_{j,k}=\kappa''_r \frac{\theta_{j,k}}{n}+\frac{1}{\theta_{j,k}}[d(r+1)\ln(1+2^{j+2})+\gamma+2\ln(j+1)],
$$
where $\gamma>0$ is a free parameter. 
We then use the chaining decomposition of 
$$Y = \sup_x \Bigg\{ \sum_{j \geq 0, k\in[r]} Z - \dE Z \Bigg\},$$
and another union bound on $j$ and $k$ to write that:
\begin{align*}
    \dP\left(Y\ge \rho^r\sum_{j\ge 0,k\in[r]}\epsilon_j t_{j,k}\right) &= 
    \dP\left( \sup_x \Big\{ \sum_{j \geq 0, k\in[r]} Z - \dE Z \Big\}\ge \rho^r\sum_{j\ge 0,k\in[r]}\epsilon_j t_{j,k}\right)\\
    &\le \dP\left( \sum_{j \geq 0, k\in[r]} \sup_x \{Z - \dE Z\} \ge \rho^r\sum_{j\ge 0,k\in[r]}\epsilon_j t_{j,k}\right)\\
    &\le \sum_{j \geq 0, k\in[r]}  \dP\left( \sup_{u_\ell, v_\ell} \{Z - \dE Z\} \ge \rho^r\epsilon_j t_{j,k}\right)\\
    &\le \sum_{j \geq 0, k\in[r]} e^{-\gamma - 2 \ln(1 + j)}.
\end{align*}
In the end, using that $\sum_{j \geq 1} j^{-2} = \pi^2 / 6$, we obtain:
$$
\dP\left(Y\ge \rho^r\sum_{j\ge 0,k\in[r]}\epsilon_j t_{j,k}\right)\le r\frac{\pi^2}{6} e^{-\gamma}.
$$
Moreover, one has
$$
\epsilon_jt_{j,k}\le\epsilon_j A_r(1+j)(d+\gamma)\left[\frac{(\rho^{-r}P_{j,k})^{1-2/r}}{n}+\frac{1}{\sqrt{nd}} + \frac{1}{n}\right],
$$
for some suitable constant $A_r$ dependent only on $r$. Fix some $k\in[r]$. Write:
\begin{align*}
\frac{1}{A_r}\sum_{j\ge 0}\epsilon_j t_{j,k}
&\le 4\frac{d+\gamma}{\sqrt{nd}}+\frac{d+\gamma}{n}
 \sum_{j=0}^{j^*(k)}\epsilon_j(2\rho^{-r}B/\epsilon_j)^{1-2/r}(1+j)\\
&\quad +\frac{d+\gamma}{n}\sum_{j>j^*(k)}\epsilon_j(2\rho^{-r}BR/B_k)^{1-2/r}(1+j)
\\
&\le 4\frac{d+\gamma}{\sqrt{nd}}+\frac{d+\gamma}{n}A'_r(\rho^{-r}B)^{1-2/r}\left\{1+(B_k/R)^{2/r}\ln(R/B_k)\right\},
\end{align*}
where $A'_r$ is another constant depending only on $r$. Since $B_k\le R$, then $(B_k/R)^{2/r}\ln(R/B_k)$ is bounded by a ($r$-dependent) constant. 
\end{proof}

We know present Corollary~\ref{corr:concentration_subgaussian}, which is a consequence of Theorem~\ref{thm:subgaussian_gen}. We consider again i.i.d. $a_i$, bounded by $R$, satisfying the subgaussian tail assumption with parameter $\rho$, and some function $\varphi$ that is 1-Lipschitz, and uniformly bounded by $B_{\varphi}$. Writing 
\begin{equation}\label{eq:defH}
H(x)=\frac{1}{n}\sum_{i=1}^na_i a_i^\top \varphi(a_i^\top x),
\end{equation}
We have the following corollary.
\begin{corollary}\label{corr:concentration_subgaussian}
Thus for $1 > \delta>0$, with probability at least $1-\delta$, it holds for some $C > 0$ that
\begin{equation}
\sup_{x\in\cS_1}\opn{H(x)}\le C\rho^3(d+\ln(1/\delta)+\ln(5\pi^2/6))\left[\frac{1+\rho^{-1}B_\varphi}{\sqrt{dn}}+\frac{1+\{\rho^{-3}R^2B_\varphi\}^{1-2/3}}{n}\right].
\end{equation}
\end{corollary}

\begin{proof}
Let us write $\varphi_3(u)=\varphi(u)-\varphi(0)$. Then $\varphi_3$ satisfies our assumptions (1-Lipschitz, $\varphi_3(0)=0$). Moreover, we can decompose matrix $H(x)-\dE H(x)$ into $M(x)+N$, where
$$
M(x)=\frac{1}{n}\sum_{i=1}^{n}\left[a_i a_i^\top \varphi_3(a_i^\top x)-\dE a_1 a_1^\top \varphi_3(a_1^\top x)\right], \qquad N=\frac{1}{n}\sum_{i=1}^{n}\left[a_i a_i^\top \varphi(0) -\dE a_1 a_1^\top\varphi(0)\right].
$$
Taking $r=2$, $\varphi_1=\varphi_2=Id$, the Theorem~\ref{thm:subgaussian_gen} gives us that
\begin{equation}
\dP\left(\opn{N}\ge C_2|\varphi(0)|\rho^2(\gamma+d)\left(1/\sqrt{d n}+1/n\right)\right)\le 2\frac{\pi^2}{6}e^{-\gamma}.
\end{equation}
Taking next $r=3$, and $B=R^2B_{\varphi}$, we obtain
\begin{equation}
\dP\left(\sup_{x\in\cS_1}\opn{M(x)}\ge C_3\rho^3(d+\gamma)\left(1/\sqrt{d n}+[\rho^{-3}R^2B_{\varphi}]^{1-2/3}/n\right)\right)\le 3\frac{\pi^2}{6}e^{-\gamma}.
\end{equation}
Combined, these two bounds give us that for all $\gamma>0$, with $C=C_2+C_3$:
\begin{equation}
\dP\left(\sup_{x\in\cS_1}\opn{H(x)}\ge C\rho^3(d+\gamma)\left[\frac{1+\rho^{-1}B_\varphi}{\sqrt{dn}}+\frac{1+\{\rho^{-3}R^2B_\varphi\}^{1-2/3}}{n}\right]\right)\le 5\frac{\pi^2}{6}e^{-\gamma}.
\end{equation}
We finally take $\gamma = - \ln\left(\frac{6\delta}{5 \pi^2}\right)$.
\end{proof}

The last step required to prove Theorem~\ref{thm:concentration_subgaussian} is to consider the supremum over $\cB(0, D)$ with an arbitrary $M_\ell$-Lipschitz function, which can be done by direct reduction:

\begin{proof}[Proof of Theorem~\ref{thm:concentration_subgaussian}]
To apply this to $\ell^\second$, defined on $\cB(0,D)$ and $M_\ell$ Lipschitz, we apply Corollary~\ref{corr:concentration_subgaussian} to $\varphi(x) = \frac{1}{M_\ell D}\ell^\second(Dx)$ (which is $1$-Lipschitz on $\cB(0, 1)$). Then, $B_\varphi = B_\ell / M_\ell D$ and the right hand side must be multiplied by $M_\ell D$.
\end{proof}

\begin{remark}
\emph{Note that there is a difference in the way Theorem~\ref{thm:sudakov_gen} and Theorem~\ref{thm:subgaussian_gen} are applied to our linear models problem. In particular, Theorem~\ref{thm:sudakov_gen} considers $\varphi_1 = \| \cdot \|^2$ and $\varphi_2 = \ell^\second$, whereas Theorem~\ref{thm:subgaussian_gen} uses $\varphi_1 = \varphi_2 = Id$ and $\varphi_3 = \ell^\second$. Theorem~\ref{thm:sudakov_gen} can be adapted to work with $r=3$, but the bound does not improve when splitting $\| \cdot \|^2$ into $Id \times Id$. Similarly, Theorem~\ref{thm:subgaussian_gen} could be used with $r=2$ and $\varphi_1 = \| \cdot \|^2 / (2R)$ (to respect the $1$-Lipschitz assumption), but in this case the bound can only be worse since the main difference is that the $\rho^3$ factor becomes $R \rho^2$, and $\rho$ is generally smaller than $R$.}
\end{remark}

\subsection{Tightness of Theorem~\ref{thm:subgaussian_gen}}
\label{app:concentration_tightness}
Consider that the $a_i$ uniformly distributed on the sphere with radius $R=\sqrt{d}$, and take for $f_k$ the identity. Such vectors can be constructed by taking vectors $A_i$ with coordinates $i.i.d.$ standard gaussian, and setting $a_i= \sqrt{d}\lVert A_i\rVert^{-1}A_i$. 
The subgaussianity parameter $\rho$ can then be taken equal to 1.

Then, using known results about maximal correlation between variables with fixed marginals~\citep[e.g.,][Section~3]{vershynin2019high}, the expectation
$
\dE \prod_{k=1}^r a_1^{\top}x_k
$
is maximized, over choices $x_k\in\cS_1$, by taking $x_1=\cdots=x_r$. We may choose $x_1=e_1$, the first unit vector, by rotational invariance, and thus the expectation is upper-bounded as:
$$
\dE \prod_{k=1}^r a_1^{\top}x_k \le \dE d^{r/2} \dE\left[ \frac{|A_i(1)|^r}{\lVert A_i\rVert^r}\right].
$$
This is of order 1, as can be shown using concentration inequalities on the deviations of $\lVert A_i\rVert$ from $\sqrt{d}$. Consider then the empirical sum $\frac{1}{n}\sum_{i\in[n]}\prod_{k\in[r]}a_i^{\top}x_k$. Choose $x_k=d^{-1/2}a_1$ for all $k\in[r]$. Then this empirical sum evaluates to
$$
\frac{1}{n}\sum_{i\in[n]}\prod_{k\in[r]}a_i^{\top}x_k = \frac{1}{n}d^{r/2} +\frac{1}{n}\sum_{i=2}^{n}\prod_{k\in[r]}a_i^{\top}a_1.
$$
The second sum can be shown to be of order 1 (conditioning on $a_1$, and then using, \emph{e.g.}, Bienaymé-Tchebitchev inequality). Thus, one cannot hope to establish concentration without extra assumptions on the data distribution unless $d^{r/2}=O(n)$. 

Contrast this with the result of Theorem~\ref{thm:subgaussian_gen}: for $R=\sqrt{d}$, $B=R^r$ and $\rho=1$, it gives that 
$$Y\le O(d^{(r/2)(1-2/r)}d/n)=O(d^{r/2}/n).$$
Thus the result is sharp for the particular example we just considered.

\section{Experiment Setting and Additional Results}
\label{app:experiments}

Some implementation details are omitted in the main text. To ease the reader's understanding, we provide these details here, along with some additional experimental results.

\textbf{Optimization problem.} We used the logistic loss with quadratic regularization, meaning that the function at node~$i$ is:
$$f_i: x \mapsto \frac{1}{m}\sum_{j=1}^m \log\left(1 + \exp(-y_{i,j} x^\top a_j^{(i)})\right) + \frac{\lambda}{2}\|x\|^2,$$
where $y_{i,j} \in \{-1, 1\}$ is the label associated with $a_j^{(i)}$, the $j$-th sample of node $i$. The local datasets are constructed by shuffling the LibSVM datasets, and then assigning a fixed portion to each worker. Then, the server subsamples $n$ points from its local dataset to construct the preconditioning dataset. To assess the suboptimality, we let the best algorithm run for more time in order to get a good approximation of the minimum error. Then, we subtract it to the running error of an algorithm to get the suboptimality at each step. 

\textbf{Tuning $\mu$.} We tune the base value of $\mu$ by starting from $0.1 / n$ and then decreasing it as long as it is stable, or increasing it as long as it is unstable. We multiply or divide $\mu$ by a factor of $1.2$ at each time. 

\textbf{Adjusting $\alpha_t$ and $\beta_t$.} We found that choosing $A_0 = 0$ and $B_0 = 1$ for SPAG is usually not the best choice. Indeed, rates are asymptotic and sequences $\alpha_t$ and $\beta_t$ converge very slowly when $\sigma_\rel$ is small, whereas we typically rarely use more than about 100 iterations of SPAG. Therefore, we start the algorithm with $A_{t_0}$ and $B_{t_0}$ with $t_0 > 0$ instead. We used $t_0 = 50$, but SPAG is not very sensitive to this choice.

\textbf{Tuning the momentum.} Figure~\ref{fig:impact_tuning} evaluates the relevance of tuning the parameters controlling the momentum of SPAG and HB-DANE. To do so, we compare the default values of $\beta = (1 - (1 + 2\mu / \lambda)^{-1/2})^2$ (for HB-DANE) and $\sigma_\rel = 1 / (1 + 2\mu / \lambda)$ (for SPAG) to values obtained through a grid search on the KDD2010 dataset with $\lambda = 10^{-7}$.
We tune HB-DANE by using a grid-search of resolution $0.05$ to test the values between $0.5$ and $1$. For $n=10^3$, theory predicts a momentum of $\beta = 0.86$ and the grid search gives $\beta = 0.85$. For $n=10^4$, theory predicts $\beta = 0.81$ and the grid search gives $\beta = 0.8$. For SPAG, we test $\sigma_\rel = 10^{-2}, \ 3 \times 10^{-3}, \ 10^{-3}$ and so on until $\sigma_\rel = 10^{-5}$ (so roughly divided by $3$ at each step). For $n=10^3$, theory predicts $\sigma_\rel = 0.005$ and the tuning yields $\sigma_\rel = 0.006$. For $n=10^4$, theory predicts $\sigma_\rel = 0.0099$ and the grid-search leads to $\sigma_\rel = 0.01$. We do not display the curves in this case ($n=10^4$) since they are nearly identical. Therefore, the grid-search always obtains the value on the grid that is closest to the theoretical value of the parameter, and the difference in practice is rather small, as can be seen in Figure~\ref{fig:impact_tuning}. This is why we use default values in the main text.

\textbf{Local subproblems.} Local problems are solved using a sparse implementation of SDCA~\citep{shalev2016sdca}. In practice, the ill-conditionned regime is very hard, especially when $\mu$ is small. Indeed, the local subproblems are very hard to solve, and it should be beneficial to use accelerated algorithms to solve the inner problems. In our experiments, we warm-start the local problems (initializing on the solution of the previous one), and keep doing passes over the preconditioning dataset until $\| \nabla V_t(x_t)\| \le 10^{-9}$ (checked at each epoch). This threshold is important because it greatly affects the performances of preconditioned gradient methods. Figure~\ref{fig:inaccurate_solving} compares the performances of SPAG, DANE and HB-DANE for different number of passes on the inner problems for the RCV1 dataset for $n=10^4$ and $\lambda=10^{-5}$. We use $\mu = 2 \times 10^{-5}$ and a step-size of $1$ for all algorithms. We first see that increasing the number of passes significantly improves the convergence speed of all algorithms. Besides, heavy-ball acceleration does not seem very efficient when local problems are not solved accurately enough. On the contrary, SPAG seems to enjoy faster rates than DANE nevertheless. It would be interesting to understand these different behaviours more in details.

\begin{figure*}[t]
\begin{subfigure}[normal]{0.5\linewidth}
    \includegraphics[width=\linewidth]{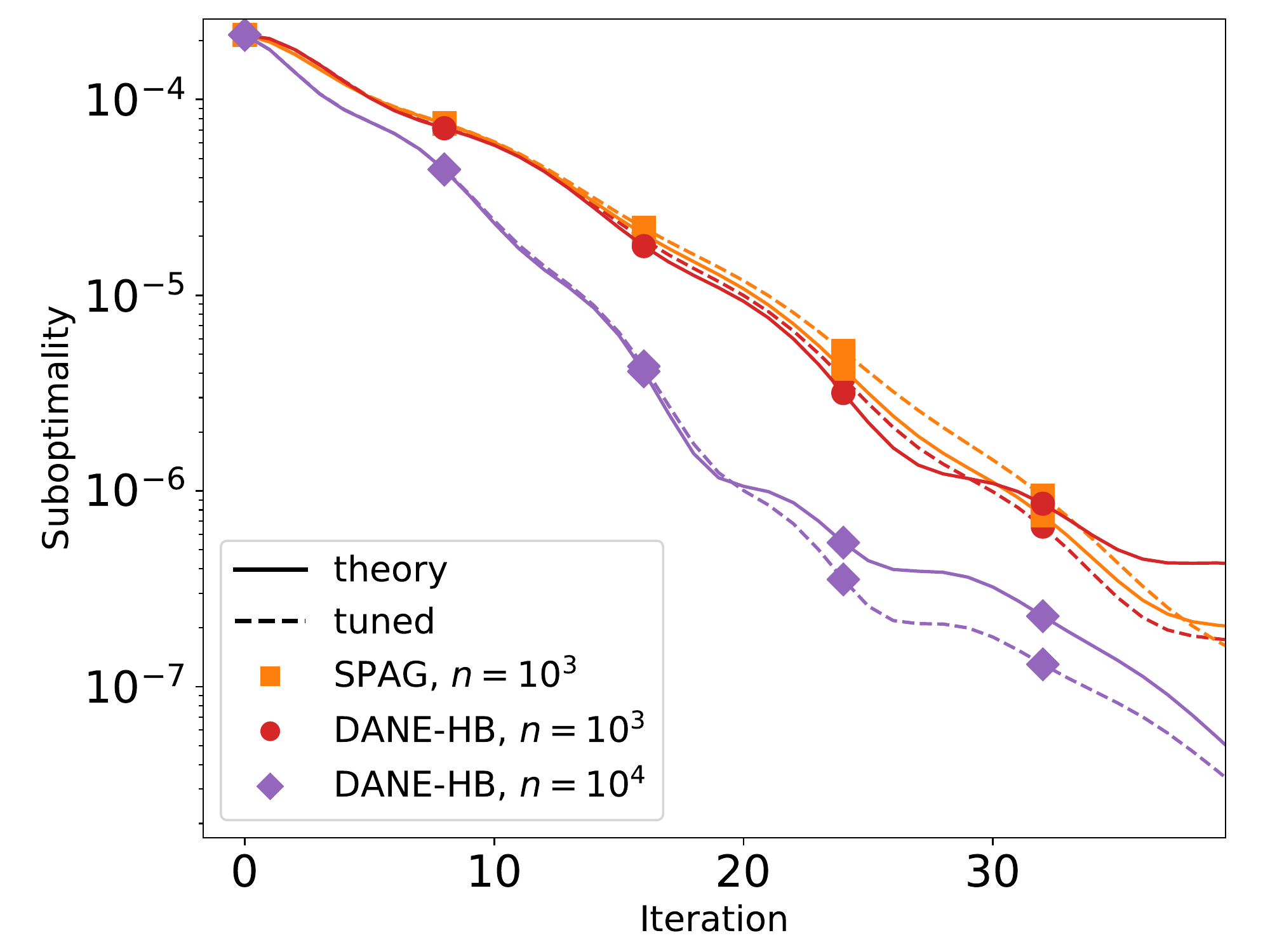}
    \caption{Impact of parameters tuning (KDD2010, $\lambda=10^{-7}$).}
    \label{fig:impact_tuning}
\end{subfigure}
\begin{subfigure}[normal]{0.5\linewidth}
    \includegraphics[width=\linewidth]{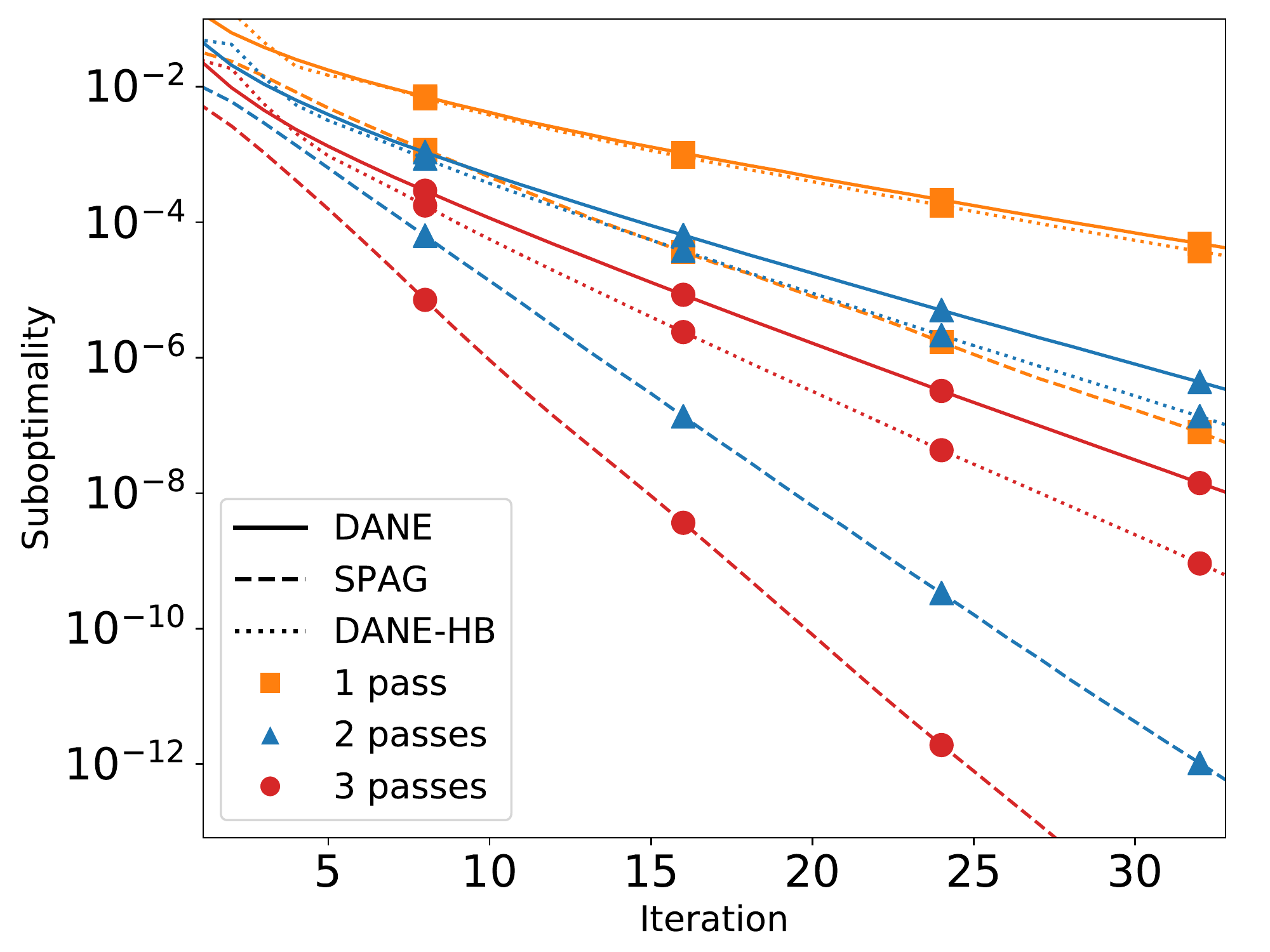}
    \caption{Impact of innacurate solving of the inner problems.}
    \label{fig:inaccurate_solving}
\end{subfigure}\\
\begin{subfigure}[normal]{0.5\linewidth}
    \includegraphics[width=\linewidth]{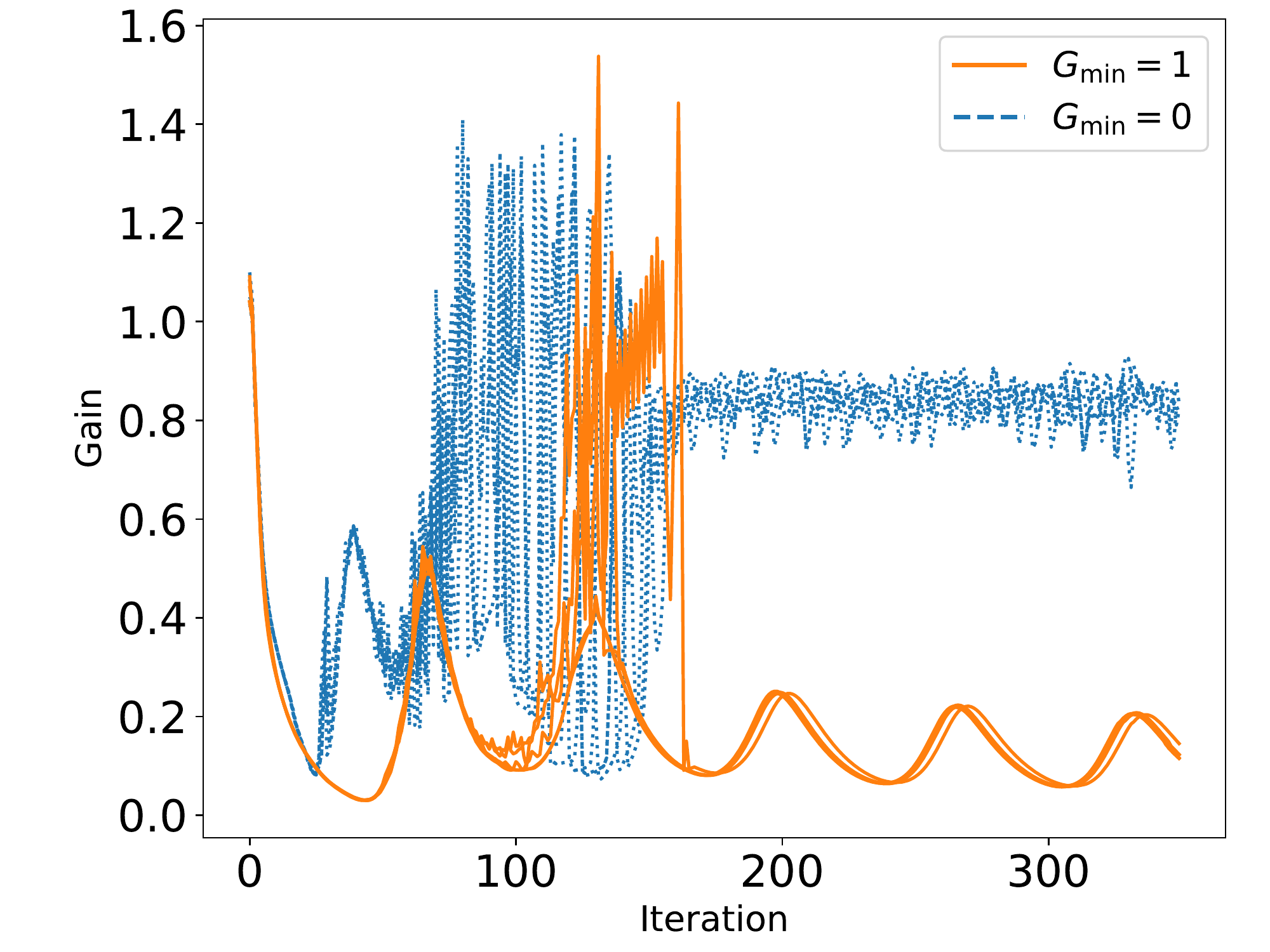}
    \caption{Impact of $G_{\min}$ on the gain.}
    \label{fig:gain}
\end{subfigure}
\begin{subfigure}[normal]{0.5\linewidth}
    \includegraphics[width=\linewidth]{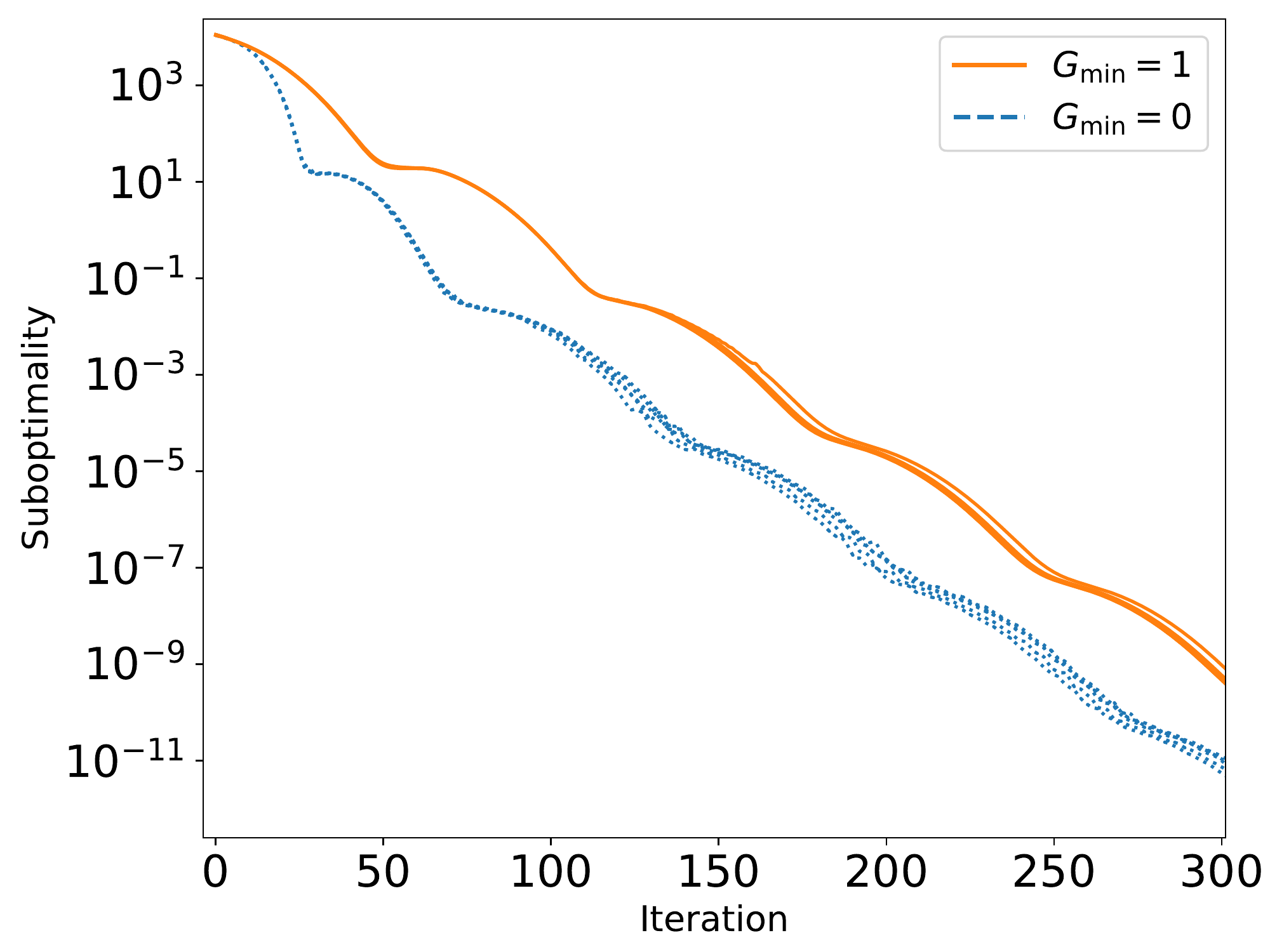}
    \caption{Impact of $G_{\min}$ on the suboptimality.}
    \label{fig:LS_comparison}
\end{subfigure}
 \caption{Impact of several implementation details.} \label{fig:added_exp}
\end{figure*}

\textbf{Gain far from the optimum.} So far, we have presented experiments with good initializations (solution for the local dataset), and argued why $G_t$ was very small in this case. Because of Lemma~\ref{lemma:LS}, one would expect that $G_t$ could be large when $x_t$ is very far from $x_*$. Yet, We see in the proof of Lemma~\ref{lemma:LS} that the Lipschitz constant of the Hessian only needs to be considered for any convex set that contains $x_{t+1}$, $v_{t+1}$, $y_t$ and $v_t$. In the case of logistic regression, the third derivative decreases very fast when far from $0$, meaning that the local Lipschitz constant of the Hessian is small when the iterates are far from $0$. In other words, the Hessian changes slowly when far from the optimum (at least for logistic regression).

We believe that this is the reason why $G_t$ can always be chosen of order 1 (smaller than 2) in our experiments, and that this holds regardless of the initialization. To support this claim, we plot in Figure~\ref{fig:gain} the values of the gain for the RCV1 dataset with $\lambda = 10^{-7}$ and 5 different $x_0$ sampled from $\cN(0, 10^3)$, the normal law centered at $0$ with variance $10^3$. We use a step-size of $0.9$ and $\mu = 2 \times 10^{-5}$. We first see that for $G_{\min} = 1$, the gain is always very low, and actually increases at some point instead of becoming lower and lower, so the fact that we were able to choose $G_t$ of order 1 in the other experiments is not linked to the good initialization. We had to choose a slightly higher $\mu$ than in the other experiments in order to satisfy the relative smoothness condition, which was not satisfied at each iteration otherwise. Since $G_t$ is small in practice and the smaller the $G_t$ the better the rate, we test SPAG with no minimum value for the gain $G_t$. The curve for the gain in this case is shown by $G_{\min} = 0$, and we see that the true gain stabilizes to a higher value, since updates are more agressive. We discuss the efficiency of this version in the next paragraph.  Note that the oscillations are not due to numerical instability or inaccurate solving of the inner problems, but rather to the fact that the step-size is slightly too big so sometimes the smoothness inequality is not verified. Yet, this does not affect the convergence of SPAG, as shown in Figure~\ref{fig:LS_comparison}. 

\textbf{Line Search with no minimum value.} Since the gain is almost always smaller than $1$, the line-search in SPAG generally only consists in checking that $G_t = 1$ works, which can be done locally. Therefore, there is no added communication cost. As discussed earlier, it is possible to allow $G_t < 1$ when performing line search, which makes SPAG slightly more adaptative at the cost of a few more line-search loops. Figure~\ref{fig:LS_comparison} presents the difference between SPAG using a line search with $G_{\min} = 0$ and $G_{\min} = 1$. The curves show the suboptimality for the runs used to generate Figure~\ref{fig:gain}. Note that we omit the cost of line search in the iteration cost (we still count in terms of number of iterations, even though more communication rounds are actually needed when $G_{\min} = 0$). We see that setting $G_{\min} = 0$ is initially slightly faster but that the rate is very similar, so that using $G_{\min} = 0$ may slightly improve iteration complexity but is not worth doing in this case. Note that suboptimality curves for different initializations are almost indistinguishable, which can be explained by the fact that the quadratic penalty term dominates and that all initializations have roughly the same norm (since $d$ is high).

\end{document}